\documentclass[reqno,11pt]{amsart}
\usepackage{amsmath,amsxtra,amsbsy,enumitem, latexsym, amsfonts, amssymb, amsthm, amscd, stmaryrd}
\setlist[itemize]{leftmargin=0.7cm}
\usepackage{hyperref}
\usepackage{graphics,epsf,psfrag}
\usepackage{bbm, dsfont}
\usepackage{mathtools}
\usepackage{xcolor}
\usepackage{float}
\usepackage{tikz}
\usepackage[headings]{fullpage}
\usepackage{colonequals} 
\usepackage{microtype}

\numberwithin{equation}{section}

\usepackage{subcaption}
\definecolor{darkred}{rgb}{0.7,0.1,0.1}

\definecolor{darkgreen}{rgb}{0.1,0.7,0.1}

\newcommand{\bbB}{{\ensuremath{\mathbb B}} }

\newcommand{\bbE}{{\ensuremath{\mathbb E}} }

\newcommand{\bbH}{{\ensuremath{\mathbb H}} }

\newcommand{\bbN}{{\ensuremath{\mathbb N}} }

\newcommand{\bbP}{{\ensuremath{\mathbb P}} }

\newcommand{\bbR}{{\ensuremath{\mathbb R}} }
\newcommand{\bbS}{{\ensuremath{\mathbb S}} }

\newcommand{\bbZ}{{\ensuremath{\mathbb Z}} }

\renewcommand{\epsilon}{\varepsilon}
\newcommand{\eps}{\varepsilon}

\newcommand{\ga}{\alpha}
\newcommand{\gb}{\beta}

\newcommand{\gep}{\varepsilon}       

\newcommand{\go}{\omega}

\newcommand{\gs}{\sigma}

\newcommand{\dd}{\mathrm{d}}

\newcommand{\cG}{{\ensuremath{\mathcal G}} }

\newcommand{\cF}{{\ensuremath{\mathcal F}} }

\newcommand{\cL}{{\ensuremath{\mathcal L}} }

\newcommand{\bP}{{\ensuremath{\mathbf P}} }

\newcommand{\bE}{{\ensuremath{\mathbf E}} }

\newcommand{\cR}{{\ensuremath{\mathcal R}} }

\newtheorem{theorem}{Theorem}[section]
\newtheorem{lemma}[theorem]{Lemma}
\newtheorem{conjecture}[theorem]{Conjecture}
\newtheorem{proposition}[theorem]{Proposition}

\newtheorem{remark}[theorem]{Remark}
\newtheorem{claim}[theorem]{Claim}
\newtheorem{definition}[theorem]{Definition}

\renewcommand{\tilde}{\widetilde}
\renewcommand{\hat}{\widehat}

\DeclareFontFamily{U}{mathx}{}
\DeclareFontShape{U}{mathx}{m}{n}{<-> mathx10}{}
\DeclareSymbolFont{mathx}{U}{mathx}{m}{n}
\DeclareMathAccent{\widehat}{0}{mathx}{"70}
\DeclareMathAccent{\widecheck}{0}{mathx}{"71}
\renewcommand{\check}{\widecheck}

\newcommand{\sumtwo}[2]{\sum_{\substack{#1 \\ #2}}} 

\definecolor{dgreen}{RGB}{30,140,60}

\newcommand{\IG}{\mathrm{IG}}

\newcommand{\R}{\mathbb{R}}

\def\ggg{{\mathcal G}}
\def\indic{{{\mathbbm 1}}}
\def\Id{{\hbox{Id}}}
\def\demi{{\frac12}}
\def\eqlaw{\stackrel{\hbox{law}}{=}}

\newtheorem{theoA}{Theorem}

\newtheorem{lemA}[theoA]{Lemma}

\newtheorem{defA}[theoA]{Definition}

\title[A random polymer approach to the VRJP]{A random polymer approach to the weak disorder phase of the Vertex Reinforced Jump Process}

\author[Q. Berger]{Quentin Berger}
\address{Universit\'e Sorbonne Paris Nord, Laboratoire d'Analyse, G\'eom\'etrie et Applications, CNRS UMR 7539, 99 Av. J-B Clément, 93430 Villetaneuse, France and Institut Universitaire de France} \email{quentin.berger@math.univ-paris13.fr}

\author[A. Legrand]{Alexandre legrand}
\address{Universit\'e Claude Bernard Lyon 1,
Institut Camille Jordan, CNRS UMR 5208, 43, Boulevard du 11 novembre 1918,
69622 Villeurbanne Cedex, France} \email{legrand@math.univ-lyon1.fr}

\author[R. Poudevigne-Auboiron]{R\'emy Poudevigne-Auboiron}
\address{Sorbonne Université, Laboratoire de Probabilités, Statistique et Modélisation, 4, place Jussieu, 75252 PARIS CEDEX, FRANCE } \email{remypoudevigne@gmail.com}

\author[C. Sabot]{Christophe SABOT}
\address{Universit\'e Claude Bernard Lyon 1,
Institut Camille Jordan, CNRS UMR 5208, 43, Boulevard du 11 novembre 1918,
69622 Villeurbanne Cedex, France and Institut Universitaire de France} \email{sabot@math.univ-lyon1.fr}

\begin{document}

\begin{abstract}
In this paper, we study the transient phase of the Vertex Reinforced Jump Process (VRJP) in dimension $d\geq 3$.
In \cite{SZ19}, the authors introduce a positive martingale and show that the VRJP is recurrent if and only if that martingale converges to~$0$.
On $\mathbb{Z}^d$, $d\ge 3$, with constant conductances $W$, 
it can be shown that there is a critical value $0<W_c(\bbZ^d)<\infty$, such that the martingale converges to $0$ if $W<W_c(\bbZ^d)$ or to a positive limit if $W>W_c(\bbZ^d)$. 
On the other hand, the VRJP martingale can be interpreted as the partition function of a non-directed polymer with a very specific $1$-dependent random potential. 
In this paper, we focus on
the question of the $L^p$ integrability of the VRJP martingale, which is related to the (diffusive) behavior of the VRJP.
First, taking inspiration from the work of Junk~\cite{J22cmp} for directed polymers in $\mathbb{Z}^{1+d}$, we prove that on the half-space $\mathbb{H}_d$ of $\mathbb{Z}^d$, for all $W>W_c(\mathbb{H}_d)$ there is some $\delta>0$ such that the VRJP martingale is in $L^{1+\delta}$.
Second, we prove that, in dimension $d\geq 4$, the VRJP martingale is in $L^{p}$ for all $p>1$ above the \textit{slab critical point} $W_c^{\rm slab} (\mathbb{Z}^d) = \lim_{m\to\infty} W_c(\mathbb{Z}^{d-1} \times \llbracket -m,m\rrbracket)$.
We also propose some related conjectures.
\end{abstract}
\maketitle

\section{Introduction}
\label{sec_intro}

The Vertex Reinforced jump Process (VRJP) is a continuous-time reinforced process living on the vertices of a weighted non-directed graph. It was proposed by Werner and first investigated by Davis and Volkov (\cite{dv1}). It has received a considerable attention in the last 10 years, after it was shown in \cite{ST15} that it is closely related to the Edge Reinforced Random Walk on the one hand, and to the $\bbH^{2|2}$ supersymmetric hyperbolic sigma model (\cite{zirnbauer91, DSZ10}) on the other hand. In this paper we investigate the $L^p$ integrability of a martingale intimately related to the VRJP. As an introduction, we motivate this question by the case of the graph $\bbZ^d$, although some of our results deal with the case of the half-space or of so-called slabs. 

\subsection{Definition of the VRJP}

Let us first remind the definition of the VRJP. Consider $\ggg=(V,E)$ a non-directed connected graph, with finite degree at each vertex and let $(W_{e})_{e\in E}$ be a set of positive conductances on the edges. The VRJP, in linear time scale, is the continuous time process $(Y_s)_{s\ge 0}$ with state space~$V$, which conditionally on the past at time $s$, if $Y_s=x$, jumps to vertex $y$ with rate
\begin{eqnarray}\label{def_VRJP}
\indic_{(x,y)\in E}  W_{x,y} (1+L_y(t)),
\end{eqnarray}
where 
$$
L_y(s)=\int_{0}^s \indic_{Y_u=y} \, du
$$ 
is the local time of the process. 

The VRJP was introduced as a continuous-time counterpart of the Vertex Reinforced Random Walk, but it is in fact related to the linearly Edge Reinforced Random Walk in the following way: if $(W_e)$ are taken random as independent Gamma random variables with positive parameters $(\alpha_e)$, then the annealed law of the discrete time process associated with the VRJP is a ERRW with initial weights $(\alpha_e)$ (see \cite[Theorem~1]{ST15}). 
Besides, it was also shown in \cite[Theorem~2]{ST15} that on finite graphs and after a proper time change, the VRJP is partially exchangeable and can be represented as mixture of Markov Jump Processes. 
The mixing measure is explicit and is in fact the distribution of the first marginal of the $\bbH^{2|2}$ supersymmetric hyperbolic sigma model, introduced and investigated in \cite{DSZ10,zirnbauer91}. 
The relation with the $\bbH^{2|2}$ model was instrumental to understand the asymptotic behavior of the VRJP and ERRW. However, we will not need it directly in this paper, and we will rather focus on a different, although closely related, representation of the VRJP by a random Schr\"odinger operator, which we now describe.

\subsection{The random Schr\"odinger operator associated with the VRJP}

The random Schr\"odinger operator representation was introduced in \cite{STZ17} for finite graphs and in \cite{SZ19} for infinite graphs. 
Let us focus for now on the $\bbZ^d$ case: the graph $\ggg$ is the lattice graph, i.e. $V=\bbZ^d$ and $\{x,y\}\in E$ if $\vert x-y\vert_1=1$. We simply write $x\sim y$ if $\{x,y\} \in E$. We also take constant conductances, $W_{x,y}=W>0$ for all $x\sim y$. 
We denote by $P$ the incidence operator on $\bbZ^d$, 
i.e. for all $f\in \bbR^V$, 
$$
P f(x)= \sum_{y\sim x}  f(y), \;\;\; \forall x\in V.
$$
If $(\beta_x)_{x\in V}$ is a function on the vertices, called potential below, we define
$$
H_\beta=\beta -W P,
$$
where $\beta$ is the operator of multiplication by the potential $\beta$. In \cite{STZ17,SZ19}, a random potential $(\beta_x)_{x\in V}$ is constructed on any locally finite graph with conductances. It can be characterized by its Laplace transform which in the $\bbZ^d$ case has the following form
\begin{equation}\label{Laplace_intro}
\bbE\left[ e^{-\demi\left<\lambda, \beta\right>}\right] ={\frac{e^{ - W\sum_{x\sim y}   (\sqrt{1+\lambda_{x}} \sqrt{1+\lambda_y}-1) }}{ \prod_{x\in V}  \sqrt{1+\lambda_x}}},
\end{equation}
for all non-negative function $\lambda\in \bbR_+^V$ with finite support.  The random potential $\beta$ has the following properties:
\begin{itemize}
\item
For any finite subset $V'\subset V$, the restriction $(H_\beta)_{| V'\times V'}$ is positive definite.
\item
The random potential is $1$-dependent, \textit{i.e.}\ for any subsets $V_1$ and $V_2$ that are such that $\hbox{dist}_\ggg(V_1,V_2)\ge 2$, then $\beta_{|V_1}$ and $\beta_{|V_2}$ are independent.
\end{itemize}
We collect below in Section~\ref{sec_preliminaries} other useful properties of the potential~$\beta$ and of the operator~$H_\beta$. In particular, the law of the potential~$\beta$ has some form of stability by restriction and conditioning, which are instrumental in the present paper. Another key property is a convex monotonicity discovered by Poudevigne \cite{Poudevigne24}, see Section~\ref{sec_preliminaries} for details.

\subsection{The martingale and polymer approach}
Let us now introduce a martingale associated with the potential $\beta$ on an infinite graph $\ggg$. Consider first an increasing sequence of finite subsets $V_n\subset V$, which exhaust $V$, \textit{i.e.}\ $V_n\subsetneq V_{n+1}$ and $\cup_{n\ge 0} V_n =V$.  
We define a positive function $(\psi_n(x))_{x\in V}$ given by the partition function of a non-directed polymer in random potential $-\ln \beta$ starting at $x$: we let $\psi_n(x)=1$ if $x\notin V_n$ and for $x\in V_n$
\begin{equation}\label{polymer_intro}
\psi_{n}(x)=\sum_{\sigma \colon x \xrightarrow{V_n} V_n^c}
 W^{|\sigma|}\prod_{i=0}^{\vert \sigma\vert -1} \beta_{\sigma_i}^{-1} \,,
\end{equation}
where the sum runs on all finite (nearest-neighbor) paths $\sigma=(\sigma_0, \ldots, \sigma_{\vert \sigma \vert})$ on the graph which go from $x$ to $V_n^c$, \textit{i.e.}\ such that $\sigma_0=x$, $\sigma_i\in V_n$ for $0\le i\le \vert\sigma\vert-1$ and $\sigma_{\vert \sigma \vert}\in V_n^c$. 
The fact that $(H_\beta)_{V_n\times V_n}$ is positive definite implies that this sum is finite (see Section~\ref{sec_preliminaries}); note in particular that $(H_{\beta}\psi_n)(x) =0$ for all $x\in V_n$. 
Consider now the filtration $(\cF_{n})$ defined by
$$
\cF_{n}=\sigma\{\beta_x, \; x\in V_n\}.
$$
Then~\cite[Theorem~1]{SZ19} proves that for all $x\in V$, $(\psi_{n}(x))_{n\geq 0}$ is a positive $(\cF_{n})$-martingale. Hence, the partition function of the non-directed polymer in potential $-\ln\beta$ has the same martingale property as for directed polymers in random environment; we refer to Section~\ref{sec_DirPol} for a more detailed comparison.
Note however that, in the present case, the martingale property is a consequence of the specific form of the law of the potential, it is in fact related to the restriction/conditioning property of that law (whereas it is directly embedded in the i.i.d.\ structure of the environment for directed polymers). 
We define
$$
\psi_\infty(x)=\lim_{n\to \infty} \psi_n(x),
$$
the a.s.\ limit of the positive martingale. Note that, by construction, $\psi_\infty$ satisfies $H_\beta(\psi_\infty)=0$.

\begin{remark}
  In the following, we will refer to the $\bbZ^d$-VRJP martingale when we consider the definition~\eqref{polymer_intro} of $\psi_n$ on the graph $V=\bbZ^d$, with the exhaustion $V_n = \llbracket -n,n \rrbracket^{ d}$ and $x=0$.
\end{remark}


\subsection{Phase transition: weak vs. strong disorder}

The limit of the martingale~$\psi_\infty$ enters into the representation of the VRJP as a mixture of Markov jump processes on infinite graph (see Section~\ref{sec_preliminaries} for details). In particular, $\psi_\infty$ characterizes the recurrence/transience of the VRJP. Indeed, we have the following dichotomy, see \cite[Proposition~3]{SZ19}:
\begin{itemize}
\item
either $\psi_{\infty}(x)=0$ a.s.\ for all $x$ and the VRJP is a.s.\ recurrent;
\item
or $\psi_{\infty}(x)>0$ a.s.\ for all $x$ a.s. and the VRJP is a.s.\ transient.
\end{itemize} 
It is of course enough to look at $\psi_n(0)$ and we write $\psi_n=\psi_n(0)$ for simplicity. Translated into polymer terminology, $\psi_\infty=0$ corresponds to the strong disorder phase and $\psi_\infty>0$ to the weak disorder phase. Besides, if $(\psi_n)_{n\geq 0}$ is bounded in~$L^3$, then the time changed VRJP satisfies an invariance principle (see the end of Section~\ref{sec_beta_inifinite} for a justification).

Poudevigne's monotonicity principle \cite[Theorem~5]{Poudevigne24} implies that for any positive convex function $f$,  $\bbE(f(\psi_n))$ is a non-increasing function of the parameter $W$. 
This implies in particular that we can define the critical point
\[
  W_c=W_c(\bbZ^d) = \inf\big\{ W, \; \psi_{\infty} >0 \text{ a.s.} \big\} \in [0,\infty ] \,,
\]
which is such that $\psi_\infty>0$ a.s.\ for $W >W_c$ and  $\psi_\infty=0$ a.s. for $W<W_c$. 

Now, it is known that in dimension $d=1,2$ the VRJP is recurrent for all $W$, hence $W_c(\bbZ^d)=\infty$; this is proved in \cite{dv1} for $d=1$ (and \cite{Pemantle} for the ERRW) and \cite{Sabot21,Kozma-Peled} for $d=2$ (and \cite{SZ19} for the ERRW). 
In dimension $d\ge 3$ however, the localization/delocalization result of \cite{DS10,DSZ10} implies that $0<W_c(\bbZ^d)<\infty$.
Additionally it is know (\cite{Rapenne}) that the martingale is uniformly integrable for $W>W_c(\bbZ^d)$.



\section{Statement of the results}\label{results}

In this paper we are interested in the $L^p$ integrability of the martingale in the weak disorder phase. This is motivated by its relation with the asymptotic behavior of the VRJP and by the fact that it is an important question in the case of directed polymer models. 

Our first result consider the half-space $\bbH_d=\{x\in \bbZ^d, \;\; x_d\geq 0\}$, which allows us to write a renewal type equation for the VRJP martingale, making it closer to that of the directed polymer. 
We prove that for $d\ge 3$ and for all $W>W_c(\bbH_d)$, the martingale $(\psi_n)_{n\geq 0}$ is bounded in $L^{1+\delta}$ for some $\delta>0$. 
We conjecture that the same result holds for~$\bbZ^d$ itself.

Our second result considers the full-space $\bbZ^d$, in dimension $d\geq 4$. 
By comparing $\bbZ^d$ to a stack of ``slabs'', we prove that the $\bbZ^d$-VRJP martingale $\psi_n$ is bounded in $L^p$ for all $p\geq 1$ above the so-called \textit{slab critical point} $W_c^{\rm slab} = \lim_{m\to\infty} W_c (\bbZ^{d-1} \times \llbracket -m,m \rrbracket)$. 
We conjecture that this also holds above the $\bbZ^d$ critical point.

\subsection{Main results I: moments of order \texorpdfstring{$1+\delta$}{larger than 1} on the half space}

Our first result deals with the half-space of $\bbZ^d$, namely
\[
  V:=\bbH_d = \{ x=(x_1,\ldots, x_d) \in \bbZ^d, x_d \geq 0\} = \bbZ^{d-1} \times \bbZ_+.
\]
We take constant conductances $W_{x,y}=W$ for all $x\sim y$. The potential $\beta=(\beta_x)_{x\in V}$ is defined by its Laplace transform, with the same formula \eqref{Laplace_intro} given in the introduction.

Instead of looking at the partition function of the $-\ln \beta$ polymer on finite boxes, it is more convenient to consider it on the slab of width $n$, $\bbS_{n}:=\{x\in \bbH_d, \;\; 0\le x_d\le n-1\} = \bbZ^{d-1}\times \llbracket 0,n-1\rrbracket$. 
We consider the outer boundary of the slab $\bbS_{n}$, denoted $\partial_n^+ := \{x\in \bbH_d, x_d = n \}$, and we define,
\begin{equation}
\label{def:Mn}
M_{n} = \sum_{\sigma \colon 0\xrightarrow{\bbS_{n}} \partial^+_n} W^{|\sigma|} \prod_{i=1}^{|\sigma|-1} \beta_{\sigma_i}^{-1} \,,
\end{equation}
where $\sigma:0\xrightarrow{\bbS_{n}} \partial^+_n$ means that $\sigma=(\sigma_0, \cdots, \sigma_{\vert\sigma\vert})$ is a finite path in $\bbH_d$, joining $0$ to $\partial^+_n$ and such that $\sigma_i\in \bbS_{n}$ for all $i\le \vert\sigma\vert-1$. 
In Section~\ref{sec_Hd}, we prove that the sum $M_n$ is finite and that it is the limit of partition functions on finite boxes, see Lemma~\ref{lem_Mmn}. 

Note that we have changed the name of the martingale from $\psi_n$ to $M_n$, to use similar notations as in~\cite{J22cmp} (and to highlight the fact that we work on the half-space). 
We also do not write the dependence on $W$ to lighten notation. 


The results recalled in the introduction about the $\bbZ^d$ case easily extends to the case of the half-space $\bbH_d$; details are given in Section~\ref{sec_Hd}. 
In particular, we have that $(M_n)_{n\geq 0}$ is a non-negative martingale with respect to the filtration $\cF_n =\sigma\{\gb_x, \; x\in \bbS_n\}$, and so it converges almost surely to a limit $M_{\infty}=\lim_{n\to\infty} M_n$.
A zero-one law also shows that the limit verifies either $M_{\infty} =0$ a.s.\ or $M_{\infty}>0$ a.s.
Additionally, Poudevigne's monotonicity in $W$ (see \cite{Poudevigne24}, or Section~\ref{sec_monotonicity} below) enables us to obtain the following phase transition: there is some 
\[
  W_c=W_c(\bbH_d) = \inf\big\{ W, \; M_{\infty} >0 \text{ a.s.} \big\} \in [0,+\infty] \,,
\] 
such that
\begin{equation}
\label{eq:phasetransition}
\bbP\text{-a.s.} \qquad \lim_{n\to\infty} M_n = M_{\infty}
\begin{cases}
=0 & \quad \text{ if } W<W_c \,, \\
>0 & \quad \text{ if } W>W_c \,.
\end{cases}
\end{equation}
As on $\bbZ^d$, one has $W_c=+\infty$ in dimension $d=1,2$ but $0<W_c <+\infty$ in dimension $d\geq 3$.
Besides, \cite[Theorem~1]{Rapenne} shows that $(M_n)_{n\geq 0}$ is uniformly integrable $W>W_c$; in other words, the critical point $W_c$ is also the critical point for the uniform integrability of the martingale $(M_n)_{n\geq 0}$, \textit{i.e.}
\[
W_c = \inf\big\{  W >0,\; (M_n)_{n\geq 0} \text{ is uniformly integrable}\big\} \,.
\]
Our goal is to improve this result and show that, for the half-space martingale, the martingale converges also in $L^{p}$ convergence for some $p>1$ as soon as $W>W_c$.

\begin{theorem}
\label{th:moment}
Let $d\geq 3$.
For any $W$ such that $\lim_{n\to\infty} M_n = M_{\infty}>0$ a.s., in particular for any $W>W_c$, there is some $\delta>0$ such that the convergence holds in $L^{1+\delta}$.
In other words, we have this new characterization for $W_c$:
\[
W_c = \inf\Big\{W,\; \exists \delta>0 \,, \; \sup_{n\geq 0} \bbE[ (M_{n})^{1+\delta} \big]  <+\infty \Big\} \,.
\]
The same result holds if we rather consider $M_n$ as the partition function of polymers on sequence of increasing \emph{finite} boxes $(V_n)_{n\ge 0}$ such that $\cup_n V_n=\bbH_d$, as in \eqref{polymer_intro}.
\end{theorem}

We also state some corollary of the proof of Theorem~\ref{th:moment}, which gives a contrasting property of the martingale in the strong disorder phase.

\begin{proposition}
  \label{prop_SD}
  For any $W$ such that $\lim_{n\to\infty} M_n = M_{\infty}=0$ a.s., in particular for any $W<W_c$, there is a constant $c>0$ such that for any $t>1$
  \[
  \bbP\Big( \sup_{n\geq 0} M_n \geq t \Big) \in \left[ \frac{c}{t}, \frac1t\right] \,.
  \]
\end{proposition}

\subsection{Main results II: moments of all orders above the slab critical point}

For a fixed $m \geq 1$, consider the VRJP on the symmetric slab $V^{(m)}:=\bbZ^{d-1} \times \llbracket -m,m \rrbracket$.
The results recalled in the introduction about $\bbZ^d$ also holds here.
One can associate a VRJP martingale $(\psi_n^{(m)})_{n\geq 0}$ as in~\eqref{polymer_intro}, by considering the exhaustion $V_n^{(m)} :=\llbracket -n,n \rrbracket^{d-1} \times \llbracket -m,m \rrbracket$ (and $x=0$).
Then, for any $m\geq 1$, the martingale $(\psi_n^{(m)})_{n\geq 0}$ converges almost surely to a limit $\psi_{\infty}^{(m)}=\lim_{n\to\infty} \psi_n^{(m)}$, and  
there is some critical point:
\[
  W_c^{(m)}=W_c(\bbZ^{d-1}\times \llbracket -m,m \rrbracket) = \inf\big\{ W, \; \psi_{\infty}^{(m)} >0 \text{ a.s.} \big\} \in [0,+\infty] \,.
\] 
Poudevigne's monotonicity result~\cite{Poudevigne24} (see also Section~\ref{sec_monotonicity} below) shows that $W_c^{(m)}$ is non-decreasing in $m$, and also that $W_c(\bbZ^d) \leq W_c^{(m)} <+\infty$ in dimension $d\geq 4$, for any $m\in \bbN$.
In particular, we can define the so-called \textit{slab critical point}:
\begin{equation}
  \label{slabcritic}
  W_c^{\rm slab} = W_c^{\rm slab} (\bbZ^d) := \lim_{m\to\infty} W_c^{(m)}  = \inf_{m\geq 1} W_c^{(m)} \,.
\end{equation}
Our main result here is that, above the slab critical point, the $\bbZ^d$-VRJP martingale $\psi_n$ defined in~\eqref{polymer_intro} converges in $L^p$ for all $p\geq 1$.

\begin{theorem}
\label{thm:allmoments}
Let $d\ge4$ and $W>W_c^{\rm slab}$. Then the $\bbZ^d$-VRJP martingale $(\psi_n)_{n\geq 0}$ defined in~\eqref{polymer_intro} converges in $L^p$ for all $p\ge1$. 
In other words, we have 
\[
W_{c}^{\rm slab} \ge \inf\Big\{W,\; \forall p\ge1 \,, \; \sup_{n\geq 0} \bbE[ (\psi_{n})^{p} \big]  <+\infty \Big\} \,.
\]
\end{theorem}

Let us notice that in dimension $d=3$, the simple random walk is recurrent on all the slabs $\bbZ^{d-1} \times \llbracket -m,m \rrbracket$, so in particular we have that $W_c^{(m)} = +\infty$, again by Poudevigne monotonicity result.
We still believe that for sufficiently large $W$, the martingale $(\psi_n)_{n\geq 0}$ converges in $L^p$ for all $p\geq 1$, but new techniques are needed, even though the idea of the proof of Theorem~\ref{thm:allmoments} could be useful.

\subsection{Some related conjectures}

By monotonicity, we have that the critical point in the half space and the slab critical point~\eqref{slabcritic} are larger than that of the full space, \textit{i.e.}\ $W_c(\bbH_d), W_c^{\rm slab} (\bbZ^d) \geq W_c(\bbZ^d)$.
A natural question is to know whether these critical points coincide, and we conjecture that they do (at least in dimension $d\geq 4$ for the slab critical point).

\begin{conjecture}
\label{conj:space}
The critical point for $\bbZ^d$ is the same as for $\bbH_d$ in all dimension $d\geq 1$, namely $W_c(\bbZ^d) = W_c(\bbH_d)$.

In dimension $d\geq 4$, the critical point for $\bbZ^d$ is equal to the slab critical point, namely $W_c(\bbZ^d) = W_c^{\rm slab}(\bbZ^d) = \lim_{m\to\infty} W_c (Z^{d-1} \times \llbracket -m,m\rrbracket)$.
\end{conjecture}

\noindent
Another conjecture, supported by Theorem~\ref{thm:allmoments} and Conjecture~\ref{conj:space} is that above the critical point, the martingale is bounded in $L^p$ for any $p>1$.
Let us note that by monotonicity, the property that the martingale is bounded in $L^p$ is monotonous in $W$, so that there exists some critical point $W_c^{(p)}$ for the $L^p$ convergence (see Remark~\ref{rem_Lp} for details): we therefore believe that $W_c =W_c^{(p)}$ for all $p>1$ in dimension $d\geq 3$, as is the case on trees (see~\cite{Rapenne}).

\begin{conjecture}
\label{conj:moment}
In any dimension $d\geq 3$, for any $W>W_c$, the convergence holds in $L^p$ for any $p>1$.
In other words, 
\[
W_c = \inf\Big\{W,\; \sup_{n\geq 0} \bbE[ (\psi_{n})^{p} \big]  <+\infty \Big\} =: W_c^{(p)} \,.
\]
\end{conjecture}

\subsection{Comparison with the directed polymer model}
\label{sec_DirPol}

Let us now comment on how results on the VRJP compare to those on the directed polymer model. 
We first introduce the model: consider $(S_n)_{n\geq 0}$ a simple (symmetric, nearest neighbor) random walk on $\bbZ^d$, with law denoted~$\bP$; the directed trajectory $(n,S_n)_{n\geq 0}$ is then interpreted as a directed polymer in $\mathbb{N}\times \bbZ^{d}$, \textit{i.e.}\ in dimension $1+d$.
Let also $(\go_{i,x})_{i\in \bbN,x\in \bbZ^d}$ be i.i.d.\ random variables (the environment) with law denoted $\bbP$, verifying $\lambda(u):=\log\bbE[e^{u \go}] <+\infty$ for all $u\in \bbR$.

Then, the directed polymer model is defined by the following Gibbs measures: for $n\in \bbN$ and $u >0$ (the inverse temperature), let
\[
\frac{\dd \bP_{n}^{u,\go}}{\dd \bP} (S) := \frac{1}{Z_n} \exp\Big( \sum_{i=1}^n\big(  u \, \go_{i,S_i} - \lambda(u) \big) \Big) \,,
\]
where $Z_n=Z_{n}^{u,\go}$ is the partition function of the model, that normalizes $\bP_{n}^{u,\go}$ to a probability measure.
The partition function can be written as
\[
Z_n = \bE\Big[\exp\Big( \sum_{i=1}^n u \go_{i,S_i} - \lambda(u) \Big) \Big]
= \sumtwo{\sigma \colon \text{ RW path}}{\text{of length }n} \frac{1}{2^n} \prod_{z\in \sigma} e^{u \go_z -\lambda(u)} \,, 
\] 
where the sum runs on \textit{directed} random walk paths $(\sigma_0,\ldots, \sigma_n)$ of length $n$ in $\bbN\times \bbZ^{d}$, that is $\sigma_i = (i,x_i)$ with $x_0=0$ and $|x_i-x_{i-1}|_1=1$ for all $0\leq i \leq n$. 
In particular, we see here that the partition function $Z_n$ has a similar flavor to the VRJP martingale on the slabs $\bbS_n$ defined in~\eqref{def:Mn}.

In fact, $(Z_{n})_{n\geq 0}$ is also a martingale, with respect to the filtration $\cG_n = \sigma\{ \go_{i,x}, i\leq n, x\in \bbZ^d\}$, as observed by~\cite{Bol89}.
We also have a similar phase transition as in~\eqref{eq:phasetransition} (the monotonicity being reversed): there is some $u_c \in [0,+\infty]$ such that 
\[
\bbP\text{-a.s.}\qquad 
\lim_{n\to\infty} Z_n = Z_{\infty}
\begin{cases}
=0 & \quad \text{ if } u>u_c \,, \\
>0 & \quad \text{ if } u<u_c \,.
\end{cases}
\]
The phase $\{u, Z_{\infty} =0\}$ is dubbed as the \textit{strong disorder} phase and $\{u, Z_{\infty} >0\}$ as the \textit{weak disorder} phase.
Also for directed polymers, we know that $u_c=0$ in dimension $d=1,2$ and that $u_c>0$ in dimension $d\geq 3$.
It is also known from~\cite{CV06} that the martingale is uniformly integrable for all $u<u_c$, namely $u_c= \sup\{u, \; Z_n \text{ is uniformly integrable}\}$.
We refer to~\cite{Com17} for a nice overview of the model and its phase transition, or to~\cite{Zyg24} for a more recent account.

Now, in dimension $d\geq 3$, a lot of progress has been made recently and let us review some of the results in relation with the VRJP.
First of all, contrary to what is expected for the VRJP martingale (see Conjecture~\ref{conj:moment}), in dimension $d\geq 3$ there exists a phase where weak disorder holds ($u<u_c$) but the convergence $\lim_{n\to\infty} Z_n =Z_{\infty}$ does not hold in $L^2$. More precisely, \cite{BT10,BGdH11,BS10,BS11} (see also~\cite[Theorem~B]{JL24a}) show that in any dimension $d\geq 3$,
\[
u_c > \sup\Big\{ u, \; \sup_{n\geq 0}\bbE\big[Z_{n}^2 \big] <+\infty \Big\}\,.
\]

By a series of recent works~\cite{FJ23,J22cmp,J23,J25,JL24a,JL24b,JL25}, Junk and his co-authors improved the understanding of the weak disorder phase, beyond the $L^2$ phase. 
First, in~\cite{J22cmp} Junk showed that, for a \textit{bounded} environment~$\go$ (the assumption is removed in~\cite{FJ23,JL24b}), in the weak disorder phase the martingale convergence holds in $L^p$ for some $p>1$.
After that, in~\cite{J23,JL24a,JL25,JL24b}, Junk and Lacoin obtained important results. We can summarize them in the following form:
\begin{itemize}
\item For any $u<u_c$ (weak disorder), the convergence holds in $L^p$ for any $p< p^*(u)$, where $u\mapsto p^*(u)$ is a continuous non-increasing function with $p^*(u)\geq 1+\frac2d$ (see \cite[Corollary~2.8]{JL24b}).
\item At $u=u_c$, then \textit{weak disorder} holds, $\lim_{n\to\infty} Z_n=Z_{\infty}>0$, and the convergence holds in $L^p$ for any $p< 1+\frac2d$ but not for $p>1+\frac2d$ (see \cite[Theorem~2.1]{JL24b}).
\item If $u>u_c$ (strong disorder), then not only $\lim_{n\to\infty} Z_n=0$ but $\lim_{n\to\infty} \frac1n \log Z_n <0$ (see \cite[Theorem~2.1]{JL25}); this is dubbed as \textit{very strong disorder}.
\end{itemize}

For the VRJP, even if the weak/strong disorder and localization/delocalization features are analogous, the picture should be a bit different than for directed polymers.
In fact, much less is known, but here is what one expects (recall the monotonicity in $W$ is reversed):
\begin{itemize}
  \item For any $W>W_c$ (weak disorder), the convergence $\lim_{n\to\infty} \psi_n=\psi_{\infty}>0$ should hold in $L^p$ for any $p>1$, see Conjecture~\ref{conj:moment}.
  \item At $W=W_c$, then \textit{strong disorder} should hold, \textit{i.e.}\ $\lim_{n\to\infty} \psi_n=0$.
  \item If $W<W_c$ (strong disorder), then not only $\lim_{n\to\infty} \psi_n=0$ but $\lim_{n\to\infty} \frac1n \log \psi_n <0$, which amounts to some \textit{exponential} localization. This is for instance an important conjecture to show that, in dimension $d=2$, $\lim_{n\to\infty} \frac1n \log \psi_n <0$ for any $W>0$.
\end{itemize}

All together, even if the phenomenology of the directed polymer model may be different than for the VRJP, we take inspiration from the recent works mentioned above.
In fact, our proof of Theorem~\ref{th:moment} follows the line of~\cite{J22cmp}, but with several adaptations due to the non i.i.d.\ and non directed character of the (polymer) martingale $M_n$ in our setting, see~\eqref{def:Mn}.

\subsection{Organisation of the rest of the paper}

\begin{itemize}
  \item In Section~\ref{sec_preliminaries}, we introduce some important notation and we review some definitions and properties of the VRJP (and of the $\beta$-potential), in particular extending the definition~\eqref{polymer_intro} of the martingale to infinite subsets.
  \item In Sections~\ref{sec_moments1}-\ref{sec_claims}, we prove Theorem~\ref{th:moment}. We focus on the main, general, steps of the proof (inspired by~\cite{J23}) in Section~\ref{sec_moments1}; then we deal with the technical claims that are specific to the VRJP in Section~\ref{sec_claims}.
  \item In Section~\ref{sec_allmoments}, we prove Theorem~\ref{thm:allmoments}, by a comparison argument between the VRJP on $\bbZ^d$ and on a stack of slabs $\bbZ^{d-1} \times \llbracket -m ,m \rrbracket$.
\end{itemize}

\section{Preliminary material and complements concerning the \texorpdfstring{$\beta$}{beta}-random potential}
\label{sec_preliminaries}

\subsubsection*{Some useful notation.} 

If $\cG=(V,E)$ is a countable graph and $U\subset V$ is a subset, for $x,y\in U$, we denote by
$$
\sum_{\sigma \colon x \xrightarrow{U} y} \,,
$$
the sum on finite paths $\sigma=(\sigma_0, \ldots, \sigma_{\vert \sigma\vert})$ on the graph $\ggg$ such that $\sigma_0=x$, $\sigma_i\in U$ for all $i\leq \vert \sigma \vert-1$ and $\sigma_{\vert \sigma\vert}=y$. 
We denote by $\partial^+_{U}=\{x\in U^c, \; \exists y\in U, \; x\sim y\}$ the outer boundary of $U$ and
$$
\sum_{\sigma \colon x \xrightarrow{U} \partial^+_U} \, ,
$$
the sum on finite paths $\sigma=(\sigma_0, \ldots, \sigma_{\vert \sigma\vert})$ on the graph $\ggg$ such that $\sigma_0=x$, $\sigma_i\in U$ for $i\leq  \vert \sigma\vert-1$ and $\sigma_{\vert \sigma\vert}\in \partial^+_U$.

Finally, if $N=(N_{x,y})_{x,y\in V}$ is a $V\times V$ matrix (or associated operator), we denote by $N_{U,U'}$ its restrictions to the subsets $U$ and $U'$. If $\eta=(\eta_x)_{x\in V}$ is a vector we denote by $\eta_U$ its restriction to~$U$.

\subsection{The \texorpdfstring{$\beta$}{beta}-potential on finite sets: definition, restriction and conditioning}
\label{sec_beta_inifinite}

Consider a finite set $V$ and a $V\times V$ real symmetric matrix $W$ with non-negative coefficient. We denote by $\ggg=(V,E)$ the associated graph  with positive coefficients, \textit{i.e.}\ $E$ is the set of edges $\{i,j\}$ such that $W_{i,j} > 0$. For $\beta=(\beta_i)_{i\in V} \in \bbR^V$ a function on the vertices we define
$$
H_\beta=\beta-W,
$$
where $W$ is the operator on $\bbR^V$ defined by $Wf(x)=\sum_{y\in V} W_{x,y} f(y)$ and $\beta$ is the operator of multiplication by $\beta$. 
We will call $\beta$ the ``potential'' and consider $H_\beta$ as a Schr\"odinger operator on $V$ with bond conductances $W_{i,j}$. 
We write $H_\beta>0$ to denote that $H_\beta$ is positive definite, and  we denote $G_\beta=\left(H_\beta\right)^{-1}$ its Green function in that case.
By classical results about $M$-matrices (see \cite[Chapter~6, Theorem~2.3]{Berman}), we have that $G_\beta$ has non-negative coefficients, and in fact the coefficients are positive if $\ggg$ is connected.

\subsubsection*{Law of the $\beta$-potential}

The following definition was introduced in \cite{STZ17}, and extended to the case of a boundary term~$\eta$ by Letac and Wezolowski, \cite[Theorem~2.2]{LW20}  (see also \cite[Lemma~4]{SZ19}). 
\begin{defA}\label{beta}
Let $\eta=(\eta_i)_{i\in V}\in \bbR_+^V$. Then the distribution on $\bbR^V$ given by
 \begin{align}  \label{nu_eta}
\nu^{W,\eta}_V(d\beta)
=
\frac{1}{\sqrt{2\pi}^{|V|}} \frac{\mathds{1}_{H_{\beta}>0}}{\sqrt{\det H_{\beta}}} e^{-\demi\left< 1, H_\beta 1 \right>-\demi \left<\eta, G_\beta \eta \right> +\left<\eta,1 \right>} d\beta
\end{align}
is a probability measure; the term \(1\) in the scalar products \(\left< 1,H_{\beta}1\right>\) and \(\left< \eta, 1\right>\) is to be understood as the vector with only ones, \textit{i.e.}\ $(1)_{i\in V}$.
The Laplace transform of $\nu^{W,\eta}_V$ is, for any \(\lambda\in \mathbb{R}_{+}^{V}\)
\begin{equation}
  \label{eq:laplace-nubetaweta}
\int e^{-\demi \left< \lambda,\beta \right>} \nu^{W,\eta}_V(d\beta)=e^{-\left< \eta,\sqrt{\lambda+1}-1 \right>-\sum_{i\sim j}W_{i,j}\left( \sqrt{(1+\lambda_{i})(1+\lambda_{j})}-1 \right)}\prod_{i\in V}\frac{1}{\sqrt{1+\lambda_{i}}} \,,
\end{equation}
where \(\sqrt{\lambda+1}-1\) should be considered as the vector \((\sqrt{\lambda_i+1}-1)_{i\in V}\).

As a simple consequence of the expression of the Laplace transform, we have that if $\beta = (\beta_i)_{i\in V}$ is a random variable with distribution $\nu_V^{W,\eta}$, then
\begin{itemize}
\item $\beta$ is 1-dependent, i.e.\ if $V_1,V_2$ are subsets of $V$ such that $\hbox{dist}_\ggg(V_1,V_2)\ge 2$, the restriction $\beta_{V_1}$ and $\beta_{V_2}$ are independent.
\item
$1/\beta_i$ has the inverse Gaussian law $\IG(\frac{1}{ \eta_i + \sum_j W_{i,j}}, 1)$.
\end{itemize}
\end{defA}

Let us stress that, compared to \cite{STZ17,SZ19}, we have changed $2\beta$ for $\beta$, which explains the extra $\demi$ in $e^{-\demi \left< \lambda,\beta \right>}$ in the Laplace transform~\eqref{eq:laplace-nubetaweta}, and some change in the normalizing constant in~\eqref{nu_eta}.

\begin{remark}
If $\eta=0$, we recover the Laplace transform given in the introduction, see \eqref{Laplace_intro}. 
The extra parameter $\eta$ appears naturally when we consider restriction of the law to a subset: this is why it is very convenient to consider the larger family of probabilities $\nu_V^{W,\eta}$ for $\eta\in \bbR^V_+$. 
\end{remark}

\subsubsection*{Relation to the VRJP}

The random potential $\beta$ is related to the VRJP on $V$, in the following way.
Consider $\beta=(\beta_i)_{i\in V}$ random and distributed according to $\nu_V^W$ (with boundary term $\eta=0$).
Then, conditionally on $\beta$, consider the Markov jump process $(X_t)$ starting at $i_0\in V$, and with jump rate from $i$ to $j$ (with $i\sim j$) given by
\begin{eqnarray}\label{VRJP-rep}
W_{i,j} \, \frac{G_\beta(i_0,j)}{G_\beta(i_0,i)} \,.
\end{eqnarray}
Then, the annealed law of the process $(X_t)$, \textit{i.e.} after taking the expectation with respect to $\beta$, is equal to the law of a time change of the VRJP starting at $i_0$. 
This was proved in \cite[Theorem~2]{ST15} in terms of the $\bbH^{2\vert 2}$-model and written in \cite[Theorem~3]{STZ17} in terms of $G_\beta$.

\subsubsection*{Restriction and conditioning formulas}

The family of probability distributions $\nu_V^{W,\eta}$ have a property of stability by restriction and conditioning, which will play an important role in the present paper. The following Lemma was found independently in \cite[Lemma~5]{SZ19} and \cite[Propositions~4.2 and~4.3]{LW20}.
\begin{lemA}
\label{lem:conditionalbeta} 
Assume $V$ is finite and let $U\subset V$. Under $\nu^{W,\eta}_V(d\beta)$, the following holds:
\begin{itemize}
  \item[(i)] The restriction $\beta_U$ is distributed according to $\nu^{W_{U,U},\hat\eta}_U$, where
  \begin{equation}
  \hat\eta\coloneqq \eta_U + W_{U,U^c} 1_{U^c}\,.
  \end{equation}
  \item[(ii)] Conditionally on $\beta_U$, its complementary part $\beta_{U^c}$ is distributed according to $\nu^{\check{W},\check{\eta}}_{U^c}$, where $\check{W}=(\check{W}_{i,j})_{i,j\in U^c}$ and $\check\eta\in(\bbR)^{U_c}$ are defined by
  \begin{equation}\label{eq:lem:conditionalbeta}
    \begin{split}
  \check{W} & \coloneqq W_{U^c,U^c} + W_{U^c,U}((H_\beta)_{U,U})^{-1}W_{U,U^c} \\
   \text{and}\qquad \check\eta & \coloneqq \eta_{U^c}+W_{U^c,U}((H_\beta)_{U,U})^{-1} \eta_U\,.
    \end{split}
  \end{equation}
\end{itemize}
\end{lemA}

An important remark in (ii) above is that by Schur's complement, we have 
$$
(G_\beta)_{U^c,U^c}= \left( (H_\beta)_{U^c,U^c}-W_{U^c,U}((H_\beta)_{U,U})^{-1}W_{U,U^c} \right)^{-1}= \left( \beta_{U^c}-\check W \right)^{-1}.
$$
Hence, by (ii), the restricted Green function $(G_\beta)_{U^c,U^c}$ has the same distribution as the Green function on $U^c$ with conductances $\check W$ and $\check \eta$.

\subsection{Extension of the \texorpdfstring{$\beta$}{beta}-potential to infinite sets}
\label{infinite-graph}

It is easy from the restriction property of Lemma~\ref{lem:conditionalbeta} to extend the definition of the $\beta$ potential to countable sets $V$. 
This was done in \cite[Section~2.2]{SZ19}, but we need here to slightly broaden the setup to include non locally finite graphs and boundary term $\eta$. 

Let $V$ be a countable set and $W=(W_{x,y})_{x,y\in V}$ be non-negative conductances such that $W_{x,y}=W_{y,x}$ and
\begin{equation}\label{finite_sum}
\sum_{y\in V} W_{x,y}<\infty, \;\;\; \forall x\in V.
\end{equation}
\begin{defA}\label{beta_infinite}
Let $\eta=(\eta_x)_{x\in V}\in \bbR_+^V$. There exists a unique probability distribution $\nu_V^{W,\eta}$ on $\bbR^V$ such that for all finite subset $U\subset V$, $\beta_{U}$ has distribution $\nu_{U}^{W_{U,U},\hat \eta_U}$, where
$$
\hat \eta_U=\eta_U + W_{U,U^c} 1_{U^c} \,.
$$
\end{defA}
\begin{proof}
The proof is similar to that in \cite{SZ19}. Take an increasing sequence of finite subsets $(V_n)_{n\ge 0}$ such that $\cup_n V_n=V$. Consider $\beta^{(n)}$ distributed as $\nu_{V_n}^{W_{V_n,V_n},\hat \eta_{V_n}}$. By \eqref{finite_sum} we see that $\hat \eta_{V_n}$ is finite for all $n$. By Lemma~\ref{lem:conditionalbeta}, it is easy to see that $\beta^{(n)}$ is a compatible sequence.
Hence, by Kolmogorov's extension theorem it gives a random variable $\beta\in \bbR^V$ such that $\beta_{V_n}\eqlaw \beta^{(n)}$. If $U\subset V$ is finite, then there is $n$ such that $U\subset V_n$ and it gives that $\beta_U$ has law $\nu_U^{W_{U,U}, \hat \eta_U}$ by restriction, as expected.
\end{proof}

\subsubsection*{Green function and martingale}

The martingale which is at the core of the paper was introduced in order to extend the representation of the VRJP in \eqref{VRJP-rep} to the case of infinite graphs. Consider now $\eta=0$ and $\beta = (\beta_x)_{x\in V}$ distributed according to $\nu_V^{W}$. We define as before
$$
H_\beta=\beta-W\,,
$$
where the operator $W$ on $\bbR^V$ is defined by $Wf(x)=\sum_{y\in V} W_{x,y} f(y)$ as soon as the function $f$ verifies $\sum_{y\in V} W_{x,y} \vert f(y)\vert<\infty$ for all $x$. 
Note that by Definitions~\ref{beta} and~\ref{beta_infinite}, for any finite subset $U\subset V$ we have $\left(H_\beta\right)_{U,U}>0$. Consider now an increasing sequence $(V_n)_{n\ge 0}$ of finite subsets such that $\cup_n V_n=V$, and define  
$$
\hat G^{(n)}_\beta= \left( \left(H_\beta\right)_{V_n,V_n}\right)^{-1} 
$$ 
the Green function of the operator $H_\gb$ restricted to $V_n$.
We sometimes extend $ \hat G^{(n)}_\beta$ to $V\times V$ by setting $0$ outside of $V_n$, \textit{i.e.}\ $\hat G^{(n)}_\beta (x,y)=0$ if $x$ or $y$ belongs to $V_n^c$. 
Since $\left(H_\beta\right)_{V_n,V_n}>0$, $\hat G^{(n)}_\beta$ has the following path expansion,
\begin{equation}\label{G_paths}
\hat G^{(n)}_\beta(x,y)= \sum_{\sigma \colon x\xrightarrow{V_n} y} \frac{ \prod_{i=0}^{\vert \sigma\vert -1} W_{\sigma_i,\sigma_{i+1}}}{\prod_{i=0}^{\vert \sigma\vert} \beta_{\sigma_i}}. 
\end{equation}

We also define $(\psi_{n}(x))_{x\in V}$ by
\begin{equation}\label{eq:def:psin}
\begin{cases}
\psi_{n}(x)=1,& \;\;\; \hbox{if $x\notin V_n$}\,,
\\
H_\beta(\psi_{n})(x)=0,& \;\;\; \hbox{if $x\in V_n$}\,.
\end{cases}
\end{equation}
It is easy to see from the positivity of $(H_\beta)_{V_n, V_n}$ that there is a unique such solution and that $\psi_{n}(x)$ can be represented as the partition function~\eqref{polymer_intro} of the non-directed polymer given in the introduction: if $x\in V_n$, 
$$
\psi_{n}(x)=\sum_{\sigma \colon x \xrightarrow{V_n} \partial^+_{V_n}} \left( \prod_{i=0}^{|\sigma| -1} \frac{W_{\sigma_i,\sigma_{i+1}}}{\beta_{\sigma_i}}\right) \,.
$$
From \eqref{G_paths}, we see that we can also write for $x\in V_n$,
$$
\psi_{n}(x)=(\hat G^{(n)} \hat \eta_{V_n})(x),
$$
where $\hat \eta_{V_n}$ is the boundary term obtained by restriction of the law of $\beta$ to $V_n$ (recall Lemma~\ref{lem:conditionalbeta}), \textit{i.e.}\ $\hat \eta_{V_n}(x)=\sum_{y\in V_n^c} W_{x,y}$ for all $x\in V_n$. The following is proved in \cite[Theorem~1]{SZ19}.

\begin{theoA}\label{Th_martingale}
For all $x,y\in V$, the sequence of random variables \((\hat{G}^{(n)}(x,y))_{n\geq 0}\) is non decreasing and converges a.s.\ to 
\[
\hat G_\beta(x,y):= \lim_{n\to \infty} \hat G_\beta^{(n)}(x,y).
\]
Moreover, $\nu_V^W$-almost surely, $0<\hat G_\beta(x,y)<\infty$ and the limit does not depend on the choice of the sequence of subsets $V_n$. 

Additionally, under the probability $\nu_V^W$, for all \(x\in V\), \((\psi_{n}(x))_{n\geq 0}\) is a positive \(\cF_{n}\)-martingale with $\cF_n=\sigma\{\beta_x, \; x\in V_n\}$. It converges a.s.\ to a
random variable \(\psi(x)\), such that $\psi(x)\ge 0$ a.s., and the limit does not depend on the choice of the
increasing sequence \((V_n)_{n\geq 0}\).
Besides, the quadratic variation of the vectorial martingale \((\psi_{n}(x))_{x\in V}\) is $\hat G^{(n)}(x,y)$ and
\(\psi_{n}(x)\) is bounded in~\(L^2\) if and only if \(\bbE_{\nu_V^W}(\hat G(x,x))<\infty\).
\end{theoA}
Note that by construction, the function $\psi$ satisfies $H_\beta \psi=0$ and that $\hat G_\beta$ is a Green function in the sense that $H_\beta \hat G_\beta=\Id$. 

\subsubsection*{Relation with the VRJP}
The random variables $\hat G_\beta$ and $\psi$ enter in the representation of the VRJP on infinite graphs.
Even though it is not fully necessary in the present paper, it enlightens the relation between the potential $\beta$ and the VRJP.
Let
$$
G(x,y):=\hat G_\beta (x,y)+\frac{1}{2\gamma} \psi(x)\psi(y),\;\; x,y\in V^2
$$
where $\gamma$ is a $\mathrm{Gamma}(\demi)$ random variable, independent of $\beta$. Then, after some time change, the VRJP starting at $x_0$ is a mixture of Markov jump processes with jump rates from $x$ to $y$ given by
$$
W_{x,y} \,\frac{G(x_0,y)}{G(x_0,x)} \,.
$$

As explained in the introduction, the asymptotic behavior of the martingale $\psi_n$ is related to the recurrence or transience of the VRJP. More precisely, when the weighted graph $(\ggg, W)$ is vertex transitive, there is a zero-one law in the sense that either $\psi(x)=0$ a.s.\ for all $x\in V$, or $\psi(x)>0$ a.s.\ for all $x\in V$, see \cite[Proposition~3]{SZ19}. (We have a slightly weaker property for the half-space $\bbH_d$ but we give details in Section~\ref{sec_Hd}.) Then, \cite[Theorem~1]{SZ19} states that the VRJP is recurrent a.s.\ if $\psi=0$ and transient a.s.\ if $\psi>0$. Besides, the $L^p$ integrability of the martingale $\psi_n$ is also related to the asymptotic behavior of the VRJP. Indeed, if $\psi_n$ is bounded in $L^3$ then the VRJP satisfies a functional CLT. This is a consequence of the proof of \cite[Theorem~3]{SZ19} where it appears that the VRJP satisfies a functional CLT as soon as the limit martingale satisfies $\bbE(\psi(0)^2)<\infty$ and $\bbE(\psi(0)^{-2})<\infty$, while in \cite[Lemma~4.3]{Rapenne} it is proved that $\bbE(\psi_n(0)^{-2})=\bbE(\psi_n(0)^{3})$ for all $n$.


\subsection{Convex Monotonicity}\label{sec_monotonicity}

The following result of Poudevigne will play a central role in our arguments. We state it in the form we will need, which is in fact equivalent to \cite[Theorem~6]{Poudevigne24}.

\begin{theoA}\label{thm_monotonicity}
Let $\ggg=(V,E)$ be a countable graph, and consider $W^-$ and $W^+$ some conductances as in Section~\ref{infinite-graph} and such that, for all $x,y\in V$, $W^-_{x,y}\le W^+_{x,y}$. We denote by $\ggg_-$ and $\ggg_+$ their associated graphs. Then, there is a coupling $(\beta^-,\beta^+)$ such that $\beta^-\sim \nu_V^{W^-}$, $\beta^+\sim \nu_V^{W^+}$ such that for all $n\ge 1$ and for all $x\in V_n$ which is connected to $V_n^c$ in $\ggg_-$, 
$$
\bbE\left[ \psi^-_n(x) \mid \beta^+\right]=\psi^+_n(x),
$$
where $\psi^-$ and $\psi^+$ are the polymer partition functions associated with $(\beta^-,W^-)$ and $(\beta^+,W^+)$. 
In particular, for any non-negative convex function~$f$, we have that 
\[
  \bbE\left[ f\big(\psi^-_n(x) \big) \mid \beta^+\right] \geq f\big( \psi^+_n(x) \big) \,.
\]
\end{theoA}
Let us explain why Theorem~\ref{thm_monotonicity} is a consequence of  \cite[Theorem~6]{Poudevigne24}. Consider the graph $\tilde \cG_n$ obtained from $\cG$ by contracting all vertices of $V_n^c$ to a single vertex denoted $\star$: more precisely, let $\tilde \cG_n=(\tilde V_n:=V_n\cup\{\star\}, \tilde E_n)$ where $\tilde E_n$ is the set of edges included in $V_n$ to which we add the edges $\{i,\star\}$ if there exists $j\in V_n^c$ such that $i\sim j$. On the edge set $\tilde E_n$ we consider the conductances $(\tilde W_e)_{e\in \tilde E_n}$ such that $\tilde W_e=W_e$ if $e$ is included in $V_n$ and $\tilde W_{i,\star}=\sum_{j\sim i, j\in V_n^c} W_{i,j}$. On the graph $\tilde \cG_n$ we consider the random potential $(\tilde \beta_i)_{i\in \tilde V_n}$ with law $\nu_{\tilde V_n}^{\tilde W_n}$ and denote by $\tilde G_{\tilde \beta}(x,y)$ its associated Green function. Then, from Lemma~\ref{lem:conditionalbeta}~(i), the law of $\tilde \beta_{V_n}$ is the same as that of~$\beta_{V_n}$. Besides, it is easy to see that we have, coupling $\beta$ and $\tilde \beta$ so that $\beta_{V_n}=\tilde \beta_{V_n}$, 
$$
\psi_n(i)=\frac{\tilde G_{\tilde \beta}(\star,i)}{\tilde G_{\tilde \beta}(\star,\star)} \,.
$$
(See \cite[Lemma~3]{SZ19} for a more detailed proof.) Hence, Theorem~\ref{thm_monotonicity} above is a direct consequence of \cite[Theorem~6]{Poudevigne24}.

\subsection{The case of infinite subsets of \texorpdfstring{$\bbH_d$}{H\_d}}
\label{sec_Hd}

We fix now the graph $\ggg=(V, E)$ as the lattice graph on $\bbH_d$, \textit{i.e.}\ $V:=\bbH_d$. 
We take constant conductances $W_e=W$ for all edge $e\in E$ and define $\beta$ as the random potential with distribution $\nu_V^W$ of Definition~\ref{beta_infinite}, which coincides with the definition~\eqref{Laplace_intro} in the introduction.

\subsubsection*{Defining the slab martingale}
Recall the definition of the ``slab'' martingale~$M_n$ in~\eqref{def:Mn},
\[
  M_{n} = \sum_{\sigma \colon 0\xrightarrow{\bbS_{n}} \partial^+_n} W^{|\sigma|} \prod_{i=1}^{|\sigma|-1} \beta_{\sigma_i}^{-1} \,,
\]
where $\bbS_n = \bbZ^{d-1} \times \llbracket 0,n-1 \rrbracket$.
We now prove that the sum $M_n$ is finite (note that $\bbS_n$ is not a finite set) and that it is the limit of partition functions on finite boxes. 

More precisely, for $n,m\geq 1$, let 
$$
\bbB_{n,m}=\{x\in \bbH_d, \;\;  0\le x_d \le n-1, \; \vert x_i\vert \le m-1 \; \forall i = 1, \cdots, d-1\} \,,
$$
let $\partial^+_{n,m}=\{x\in \bbH_d\setminus \bbB_{n,m}, \; \exists y \in \bbB_{n,m}, \; x\sim y\}$ be its outer boundary, and let us define
\begin{equation}\label{M_mn}
M_{n,m} :=\sum_{\sigma \colon 0\xrightarrow{\bbB_{n,m}} \partial^+_{n,m}} W^{|\sigma|} \prod_{i=1}^{|\sigma|-1} \beta_{\sigma_i}^{-1}.
\end{equation}
In the last expression, the sum runs on finite paths $\sigma$ in $\bbH_d$ starting at $0$ and such that $\sigma_i\in \bbB_{n,m}$ for $i\le \vert \sigma\vert -1$ and $\sigma_{\vert \sigma \vert}\in \partial^+_{n,m}$.

We now have the tools to show (an improved version of) the convergence $\lim_{m\to\infty} M_{n,m} = M_n$ and some of the assertions of Section~\ref{results} concerning the $0$--$1$ law and critical parameters.
\begin{lemma}\label{lem_Mmn}
Fix $n$ a positive integer. Then, we have
$$
\lim_{m\to\infty} M_{n,m}=M_n
$$
a.s. and in $L^p$ for all $p\ge 1$.
As a consequence $M_n$ is a $\cF_n$-martingale with $\cF_n=\sigma\{\beta_x, \; x\in \bbS_n\}$ and the following limit exists 
$$ 
M_\infty=\lim_{n\to\infty} M_n \quad\text{a.s.}
$$
\end{lemma}

In other words, we somehow extend the definitions of Section~\ref{infinite-graph}, using an exhaustion of~$\bbH_d$ by an increasing sequence of \textit{infinite} subsets $V_n=\bbS_n$ such that $\cup_n V_n = \bbH_d$.
Then, the second convergence result of Lemma~\ref{lem_Mmn} is the analogue of Theorem~\ref{Th_martingale} in this context.

\begin{proof}
By Theorem~\ref{Th_martingale} applied to the sequence of finite boxes $\bbB_{m,n}$, we have that $(M_{n,m})_{m\ge 1}$ is a positive $(\cF_{n,m})_{m\geq 1}$-martingale with $\cF_{n,m}=\sigma\{\beta_x, \; x\in \bbS_{n,m}\}$. Thus, the almost sure limit $M_{n,\infty}= \lim_{m\to\infty} M_{m,n}$ exists and is non-negative. 
We now need to prove two things: that the martingale $(M_{n,m})_{m\ge 1}$ is bounded in $L^p$ for all $p>1$ and that $M_{n,\infty}=M_n$. 

The first part is a consequence of the monotonicity presented in Section~\ref{sec_monotonicity}. 
Indeed, let $z_n= (0, \cdots, 0, n)\in\partial^+_n$ and consider the subset $I_n$ of points on the segment between $0$ and $z_n$. 
Consider now $M_n^{-}$ the partition function of the polymer between $0$ and $z_n$ for conductances $W^-_{x,y}$ equal to $W$ for edges included in $I_n$ and $0$ outside. Then, by Theorem~\ref{thm_monotonicity}, for all $p\ge 1$, $\bbE((M_{n,m})^p)\le \bbE((M_n^{-})^p)$.
Besides, $\bbE((M_n^{-})^p)<\infty$ for all $p\ge 1$ since $I_n$ is a finite graph.

The last part of the proof, that $M_{n,\infty}=M_n$ a.s., is given in Appendix~\ref{appendix_A} since it is of a slightly different nature than the rest of the paper: it is based on the VRJP and is in fact equivalent to saying that the VRJP on $\bbH_d$ cannot stay an infinite time in the slab $\bbS_n$.

By the martingale property of Theorem~\ref{Th_martingale},  for $n\le n'$ and $m\le m'$ we have $\bbE(M_{n',m'}| \cF_{n,m})= M_{n,m}$. By the previous convergence it implies that $\bbE(M_{n'}| \cF_{n,m})= M_{n,m}$. Letting $m$ go to~$\infty$, we get that $\bbE(M_{n'}| \cF_{n})= M_{n}$, which proves that $(M_n)_{n\geq 0}$ is indeed a martingale.
\end{proof}

\subsubsection*{Phase transition}
We now state results on the phase transition on $\bbH_d$, which are easy extensions of the results already proved for $\bbZ^d$.  

\begin{lemma}
In dimension $d=1,2$ we have $M_\infty=0$ a.s.\ for any $W>0$.
In dimension $d\ge 3$ there is a critical parameter $W_c=W_c(\bbH_d) \in (0,+\infty)$ such that $M_\infty=0$ a.s. for $W<W_c$ and $M_\infty>0$ a.s. for $W>W_c$.
\end{lemma}

\begin{proof}
The proof of the $0$--$1$ law of \cite[Proposition~3]{SZ19} extends readily to the case of $\bbH_d$ since the graph $\bbH_d$ is invariant by all horizontal translations (\textit{i.e.}\ translation by a vector of $\bbZ^{d-1}\times\{0\}$). 
The monotonicity property of Theorem~\ref{thm_monotonicity} easily shows the existence of some critical value 
\[ W_c = \inf\{W,\; M_{\infty} >0 \text{ a.s.}\} \in [0,+\infty] \] 
such that $M_{\infty}=0$ if $W<W_c$ and $M_{\infty} >0$ if $W>W_c$.

For $d=1,2$, since the random walk on $\bbH_d$ is recurrent we can apply \cite[Theorem~4]{Poudevigne24}, which implies that $W_c=0$ since the recurrence of the VRJP is equivalent to having $M_\infty=0$. 
For $d\ge 3$, we know that $W_c >0$ since the VRJP on a graph with bounded degree is recurrent for $W$ small enough, see~\cite{ACK14,ST15}. 
Besides, the proof of delocalization in \cite{DSZ10} extends easily to the case of the half-space and it implies that the VRJP is transient for $W$ large enough; in fact, it shows that for any $p>1$, $M_n$ is bounded in $L^p$ for $W$ large enough. 
\end{proof}

\begin{remark}
  \label{rem_Lp}
  Let us stress that in dimension $d \geq 3$, there is also some critical point $W_c^{(p)}$ for the $L^p$ convergence. 
  Indeed, the monotonicity property of Theorem~\ref{thm_monotonicity} easily shows that $W\mapsto \bbE[ (M_n)^p]$ is non-increasing, so for $p>1$ we can define
  \[
  W_c^{(p)} = \inf\Big\{W,\; \sup_{n\geq 1} \bbE[(M_n)^p] <+\infty\Big\}  \,.
  \] 
  We then have that $M_{\infty}\in L^p$ for $W>W_c^{(p)}$ and $M_{\infty}\notin L^p$ for $W<W_c^{(p)}$. 
  We obviously have $W_c^{(p)} \geq W_c >0$, and the fact that $W_c^{(p)}<+\infty$ in dimension $d\geq 3$ is a consequence of the proof of delocalization in \cite{DSZ10}.
\end{remark}

\subsubsection*{Restriction and conditioning properties}

It will be useful in the remainder of this paper to extend Lemma~\ref{lem:conditionalbeta} to infinite subsets such as $U=\bbS_n$. In particular, we want to compute the law of $\beta_{\bbH_d\setminus \bbS_n}$ conditioned on $\cF_n$. 
In the following, we denote $\hat G^{(n)}$ the Green function restricted to the slab $\bbS_n$, that is 
\begin{equation}\label{G_paths_slab}
\hat G^{(n)}(x,y)=\lim_{m\to \infty} \left( (H_\beta)_{\bbB_{n,m}\times \bbB_{n,m}}\right)^{-1}(x,y)= \sum_{ \sigma \colon x \xrightarrow{\bbS_n} y} W^{\vert \sigma\vert} \prod_{i=1}^{|\sigma|-1} \beta_{\sigma_i}^{-1}.
\end{equation}
For $x,y\in \partial^+_n$, we denote by $\tilde x,\tilde y$ the neighbors of $x,y$ which lay in $\bbS_n$. For $x,y\in \bbH_d\setminus \bbS_n$, we set
\begin{equation}\label{Wcheck}
\check W_{x,y}=
\begin{cases}
W \indic_{x\sim y} \,, &\hbox{ if either $x$ or $y$ are not in $\partial^+_n$}
\\
W\indic_{x\sim y} + W^2 \hat G^{(n)}(\tilde x, \tilde y) \,, &\hbox{ if both $x$ and $y$ are in $\partial^+_n$}
\end{cases}
\end{equation}
\begin{lemma}\label{condionning_Hd}
The conductances $\check W$ satisfy the condition of Section~\ref{sec_beta_inifinite}, i.e.\ for all $x\in \bbH_d\setminus \bbS_n$ we have $\sum_{y} \check W_{x,y}<\infty$. Additionally, conditionally on $\cF_n$, $\beta_{\bbH_d\setminus \bbS_n}$ has the law $\nu^{\check W}_{\bbH_d\setminus \bbS_n}$.
\end{lemma}

\begin{proof}
Let $U$ be a finite subset of $\bbH_d\setminus \bbS_n$ and let $U_{n,m}=U\cup \bbB_{n,m}$. 
By Definition~\ref{beta_infinite}, the law of~$\beta_{U_{n,m}}$ is given by $\nu_{U_{n,m}}^{W_{U_{n,m},U_{n,m}}, \hat \eta^{(n,m)}}$ with $\hat \eta^{(n,m)}_x =\sum_{y\in \bbH_d\setminus U_{n,m}} W_{x,y}$ for $x\in U_{n,m}$.
Then, by Lemma~\ref{lem:conditionalbeta}, the law of $\beta_U$ conditioned on $\bbB_{n,m}$ is given by $\nu_{U}^{\check W^{(n,m)}, \check \eta^{(n,m)}}$ with
$$
\check W^{(n,m)}= W_{U,\bbB_{n,m}} \hat G^{(n,m)} W_{\bbB_{n,m},U} \,, \quad \hbox{ and } \quad  \check \eta^{(n,m)}= \hat \eta^{(n,m)}_{U} + W_{U,\bbB_{n,m}} \hat G^{(n,m)}  \hat \eta^{(n,m)}_{\bbB_{n,m}} \,.
$$
where we have denoted 
$$
\hat G^{(n,m)}= \left( (H_{\beta})_{\bbB_{n,m}, \bbB_{n,m}}\right)^{-1} \,,
$$
which verifies $\lim_{m\to \infty}  \hat G^{(n,m)}= \hat G^{(n)}$.
We also clearly have $\lim_{m\to \infty} \check W^{(n,m)}= \check W_{U,U}$ with the definition~\eqref{Wcheck}. 
Additionally, for the same reason as in the proof of Appendix~\ref{appendix_A}, we have for all $x\in U$ 
$$
\lim_{m\to \infty} \check \eta^{(n,m)}_x=\sum_{y\in \bbH_d\setminus (\bbS_n\cup U)} \check W_{x,y} ,
$$
since in $\hat G^{(n,m)}  \hat \eta^{(n,m)}_{\bbB_{n,m}}$ the contribution coming from the side boundary of $\bbB_{n,m}$ goes to $0$ when $m$ tends to $+\infty$.
By Definition~\ref{beta_infinite}, this implies that the law of $\beta_{\bbH_d\setminus \bbS_n}$ conditioned on $\cF_n$ is given by~$\nu^{\check W}$. 
\end{proof}

\begin{remark}\label{rem:condionning_Hd}
In the remainder of this paper, it will sometimes be useful to augment the set $\bbS_n$ with finitely many points from $\partial^+_n$ (see the definitions of $\Lambda_n$, $\cG_n$ in~\eqref{def:Rn} below). Lemma~\ref{condionning_Hd} and its proof can be adapted to that setting straightforwardly, but we do not write the details for the sake of conciseness.
\end{remark}

\section{Proof of Theorem~\ref{th:moment}: main steps}
\label{sec_moments1}

We now adapt the strategy of \cite{J22cmp} to prove Theorem~\ref{th:moment}. 
First, we show a product-type structure for the martingale $(M_n)_{n\geq 1}$,  based on a renewal formula: this is where the half-space $\bbH_d$ is really used, and enables us to use the critical \cite[Theorem~2.1]{J22cmp}, see Section~\ref{sec_renewal}.
(Note that in the directed polymer setting, this product structure is immediate due to the directedness and i.i.d.\ form of the environment.)
This enables us to show as a first step that, in the weak disorder phase, not only is the martingale uniformly integrable but $\sup_{n\geq 0} M_n$ is in fact integrable, see Section~\ref{sec_dominated}.
Then, we bound $\bbE[(M_n)^p]$ by using the overshoot of the martingale when it exceeds a threshold~$t$, see Section~\ref{sec_Lpcontrol}.
Finally, the last step consists in controlling this overshoot, see Section~\ref{sec_overshoot}: we treat the problem in a general fashion at first, postponing technical, model-dependent estimates to Section~\ref{sec_claims} --- this is the most technical part of the paper because of the intricate definition and restriction/conditioning properties of the $\beta$ potential (at the same time, we rely strongly on these specific properties).

\subsection{Renewal formulae on \texorpdfstring{$\bbH_d$}{H\_d}}\label{sec_renewal}
 
To adapt the strategy of \cite{J22cmp} we want to write a renewal type equation for $M_{k+\ell}$ for $k,\ell$ two positive integers, and decompose polymers of width $k+\ell$ into polymers from~$0$ to $x\in \partial^+_k$ and polymers from $x$ to $\partial^+_{k+\ell}$. 
However, our polymers are non directed and a polymer from $x$ to $\partial^+_{k+\ell}$ may make some excursion in $\bbS_k$. Besides, the potential $\beta_{\bbS_k}$ and $\beta_{\bbS_{\ell}\setminus \bbS_k}$ are not independent. 
However, thanks to the specific form of the potential, these two difficulties combine well and we can write a type of renewal equation at the price of a modification of the conductances on the slab from level $k$ to level $k+\ell$. 

We denote by $\bbS_{k,k+\ell}$ the slab between level $k$ et $k+\ell-1$,
$$
\bbS_{k,k+\ell}=\{x\in \bbH_d, \; k\le x_d\le k+\ell-1\}.
$$
Recall the definition of $\check W$ in \eqref{Wcheck}. For $x\in \partial_k^+$, we denote by $\check M_\ell (x)$ the partition function of polymers in the slab $\bbS_{k,k+\ell}$ and with conductances $\check W$,
\begin{equation}\label{eq:def:checkM:1}
\check M_\ell (x)= \sum_{ \sigma \colon x\xrightarrow{\bbS_{k,k+\ell}} \partial_{k+\ell}^+}  \prod_{i=1}^{|\sigma|-1} \check W_{\sigma_i, \sigma_{i+1}} \beta_{\sigma_i}^{-1} \,,
\end{equation}
where the sum runs on finite paths for the graph associated with conductances $\check W$ (which is the lattice graph with extra edges between any two points of $\partial^+_k$).
By construction of $\check W$, we see that we also have
\begin{equation}\label{eq:def:checkM:2}
\check M_\ell (x)= \sum_{ \sigma \colon x \xrightarrow{\bbS_{k+\ell}} \partial_{k+\ell}^+}   W^{\vert\sigma\vert}  \prod_{i=1}^{|\sigma|-1} \beta_{\sigma_i}^{-1} \,,
\end{equation}
where the paths are allowed to venture inside~$\bbS_k$.
For $z\in \partial^+_k$, we denote by $\alpha_z$ the probability that a polymer in $\bbS_k$ started from $0$ reaches $\partial_k^+$ at the site~$z$, that is
$$
\alpha_z  \coloneqq \frac{1}{M_k} \sum_{ \sigma \colon 0\xrightarrow{\bbS_{k}} z}   W^{\vert\sigma\vert}  \prod_{i=1}^{|\sigma|-1} \beta_{\sigma_i}^{-1} \,.
$$
With these definitions, we then have our renewal product-type formula:
\begin{equation}
\label{eq:decompoMn}
M_{k+\ell} =  M_k \sum_{z \in \partial^+_k} \alpha_z \check{M}_{\ell}(z) \,.
\end{equation}
(The analogous decomposition is immediate in the context of directed polymers, see e.g.\  \cite[\S2.1.1, Eq.~(2.3)]{Com17}.)

Recall that $\cF_k=\sigma\{\beta_x, \; x\in \bbS_k\}$.
Then, by Lemma~\ref{condionning_Hd}, conditionally on $\cF_k$, the law of $(\beta_x)_{x\in \bbH_d\setminus \bbS_k}$ is given by $\nu^{\check W}$. Hence, conditionally on $\cF_k$, $\check{M}_{\ell}(z)$ is indeed the polymer partition function in $\bbS_{k,k+\ell}$, starting from $z$, with the weights $\check W$. Since $\check W_{x,y}\ge W\indic_{x\sim y}$, the monotonicity result Theorem~\ref{thm_monotonicity} gives the following crucial inequality: 
for any non-negative convex function $f$, we have
\begin{equation}\label{convex_Sk}
  \bbE\left( f\big( \check M_\ell(z) \big) \; \big| \; \cF_k\right) \leq \bbE(f(M_\ell))  \, .
\end{equation}

%
%
%
%
%
%
%

\subsection{Step 1. Integrability of \texorpdfstring{$\sup_{n\geq 0} M_n$}{sup M\_n} in the weak disorder phase}
\label{sec_dominated}

We start by proving the following result, which is the analogue of \cite[Th.~1.1-(i)]{J22cmp} in the context of the VRJP.

\begin{proposition}
\label{prop:domination}
Assume that $\lim_{n\to\infty} M_n = M_{\infty} >0$, which holds in particular if $W>W_c$.
Then we have that
\[
\bbE\Big[ \sup_{n\geq 0} M_n \Big] <+\infty \,.
\]
\end{proposition}

This shows that for any $W>W_c$, the martingale $(M_n)_{n\geq 0}$ is not only uniformly integrable (as proven in~\cite[Theorem~1]{Rapenne}), but it is dominated by an integrable random variable.

\begin{proof}
The proof mostly relies on the inequality from \cite[Theorem~2.1-(i)]{J22cmp}, together with our renewal decomposition formula~\eqref{eq:decompoMn} and convex-monotonicity~\eqref{convex_Sk}.
Let us reproduce here Theorem~2.1-(i) of~\cite{J22cmp} in our context, for convenience. 
If $(M_n)_{n\geq 0}$ is a non-negative $(\cF_n)$-martingale such that, for any convex function $f:\bbR_+\to \bbR$ we have a.s.
\begin{equation}
\label{eq:condJunk}
\bbE\bigg[ f\Big( \frac{M_{k+\ell}}{M_k} \Big) \;\Big|\; \cF_k\bigg] \leq \bbE\big[ f(M_{\ell}) \big] \qquad \text{ for all } k,\ell\geq 0 \,,
\end{equation}
then, if $M_{\infty}=\lim_{n\to\infty} M_n$ a.s., having $\bbP(M_{\infty}>0)>0$ implies that $\bbE[\sup_n M_n] <+\infty$.

We therefore only have to prove that the condition~\eqref{eq:condJunk} is satisfied in our case. 
We use the decomposition~\eqref{eq:decompoMn}: noticing that $\sum_{z \in \partial^+_k} \alpha_z = 1$ and applying Jensen's inequality, we get that for a non-negative convex function~$f$
\[
f\Big( \frac{M_{k+\ell}}{M_k} \Big) = f\Big(\sum_{z \in \partial^+_k} \alpha_z \check{M}_{\ell}(z)\Big) \leq \sum_{z \in \partial^+_k} \alpha_z f\big(\check{M}_{\ell}(z)\big) \,.
\]
Now, since the $\alpha_z$ are $\cF_k$-measurable, we obtain that 
\begin{equation}
  \label{convconv}
  \bbE\bigg[ f\Big( \frac{M_{k+\ell}}{M_k} \Big) \;\Big|\; \cF_k\bigg]  \leq \sum_{z \in \partial^+_k} \alpha_z \bbE\big[ f\big(\check{M}_{\ell}(z)\big) \;\big|\; \cF_k \big] \leq \sum_{z \in \partial^+_k} \alpha_z \bbE[f(M_{\ell})] \,,
\end{equation}
where we have used~\eqref{convex_Sk} for the last inequality.
Since $\sum_{z \in \partial^+_k} \alpha_z = 1$, this concludes the proof of~\eqref{eq:condJunk} and thus of Proposition~\ref{prop:domination}.
\end{proof}

\subsection{Step 2. Control of \texorpdfstring{$\bbE[(W_n)^p]$}{E[W\_n\^{}p)]} via its overshoot}
\label{sec_Lpcontrol}
We now show Theorem~\ref{th:moment} using the strategy of~\cite{FJ23,J22cmp}.
For $t>1$, let us define the stopping time
\[
\tau= \tau_t = \inf\{n\geq 1, M_n >t\} \,.
\]

\begin{remark}
\label{rem:tailtau}
We stress right away that, since $\{\sup_{n\geq 0} M_n > t\} = \{\tau_t<\infty\}$, Proposition~\ref{prop:domination} gives that, for any $W>0$ such that $\lim_{n\to\infty} M_n = M_{\infty} >0$ a.s.\
\[
\bbE\Big[ \sup_{n\geq 0} M_n \Big] = \int_0^{\infty} \bbP( \tau_t <+\infty ) \dd t <+\infty \,. 
\]
As a consequence, using that $t\mapsto \bbP( \tau_t <+\infty)$ is non-increasing, we have $\lim_{t\to\infty}t\bbP( \tau_t <+\infty) =0$.
In particular, for any $\gep>0$ there is some $t_{\gep} >1$ such that $\bbP( \tau_t <+\infty) \leq  \gep t^{-1}$ for all $t\geq t_{\gep}$.
\end{remark}

Let us fix $p \in (1,2]$ and $t>1$ (specified later in the proof) and write
\[
\bbE\big[ (M_n)^p \big] \leq t^p + \bbE\big[ (M_n)^p \indic_{\{\tau_t \leq n\}}\big] = t^p + \sum_{k=1}^n  \bbE\big[ (M_n)^p \indic_{\{\tau_t = k\}}\big] \,.
\]
Now, using again the decomposition~\eqref{eq:decompoMn}, and since the function $x\mapsto x^p$ is convex, we get, exactly as in~\eqref{convconv},
\[
\bbE\bigg[ \Big(\frac{M_n}{M_k}\Big)^p \;\Big|\; \cF_k \bigg] \leq \sum_{z\in \partial^+_{k}} \alpha_z \bbE\big[ \check{M}_{n-k}(z)^p \;\big|\; \cF_k\big] \leq \sum_{z\in \partial^+_{k}} \alpha_z \bbE\big[ (M_{n-k})^p \big]  = \bbE\big[ (M_{n-k})^p \big]\,.
\]
All together, conditioning first by $\cF_k$, we get
\[
 \bbE\big[ (M_n)^p \indic_{\{\tau_t = k\}}\big] = \bbE\Big[ (M_k)^p \indic_{\{\tau_t = k\}} \bbE\big[ (M_n/M_k)^p \mid \cF_k \big] \Big]
 \leq  \bbE\big[ (M_{n-k})^p \big] \bbE\big[ (M_k)^p \indic_{\{\tau_t = k\}} \big] \,.
\]
Since $\bbE[ (M_{n-k})^p] \leq \bbE[(M_n)^p]$ because $(M_n^p)_{n\geq 0}$ is a submartingale, we end up with the following inequality:
\begin{equation}
\label{ineq:Lp}
\bbE\big[ (M_n)^p \big] \leq t^p + \bbE\big[ (M_n)^p \big] \sum_{k=1}^{n} \bbE\big[ (M_k)^p \indic_{\{\tau_t = k\}} \big]  \leq t^p + \bbE\big[ (M_n)^p \big] \bbE\big[ (M_{\tau})^p \indic_{\{\tau_t <+\infty\}} \big]   \,.
\end{equation}

We then rely on a technical lemma, which controls the second moment of the so-called overshoot of $(M_n)_{n\geq 0}$ at the instant $\tau_t$.
This lemma is the core of the method; we postpone its proof to the next subsection.
\begin{lemma}
\label{lem:oveshoot}
There exists a constant $C=C_W>1$ such that, for any $t>1$,
\[
\bbE\big[ (M_{\tau})^2 \indic_{\{\tau_t <+\infty\}} \big] \leq C t^2 \, \bbP(\tau_t <+\infty) \,.
\]
\end{lemma}

Let us now conclude the proof of Theorem~\ref{th:moment} thanks to this lemma. First, notice that Lemma~\ref{lem:oveshoot} and the conditional H\"older inequality $\bbE[ (M_{\tau})^p \mid \tau_t <+\infty ] \leq \bbE[ (M_{\tau})^2 \mid \tau_t <+\infty ]^{p/2}$ imply that
\[
\bbE\big[ (M_{\tau})^p \indic_{\{\tau_t <+\infty\}} \big] \leq C t^p \, \bbP(\tau_t <+\infty) \,,
\]
where $C$ is the constant in Lemma~\ref{lem:oveshoot} (note that $C^{p/2}\leq C$).

Let us now fix $\gep:=\frac{1}{4C}$ and also $t=t_{\gep}$ from Remark~\ref{rem:tailtau} such that $C t \bbP(\tau_t<\infty) \leq \frac{1}{4}$. Then, for such fixed $t$, we have $\bbE[ (M_{\tau})^p \indic_{\{\tau_t <+\infty\}}] \leq  \frac14 t^{p-1}$.
Hence, since $t=t_{\gep}$ is fixed, one can choose $p = p_{\gep}$ close enough to $1$ so that $\bbE\big[ (M_{\tau})^p \indic_{\{\tau_t <+\infty\}} \big] \leq \frac12$.
Going back to~\eqref{ineq:Lp}, this gives that
\[
\bbE\big[ (M_n)^p \big] \leq t^p + \frac12  \bbE\big[ (M_n)^p \big] \,,
\]
so that $\bbE\big[ (M_n)^p \big] \leq 2 t^p$, uniformly in $n$.
This concludes the proof of Theorem~\ref{th:moment} since $t$ is a fixed quantity.
\qed

\begin{proof}[Proof of Proposition~\ref{prop_SD}]
Recalling that $\tau_t\coloneqq \inf\{n\ge1, M_n>t\}$ and that we assume that $\lim_{n\to\infty} M_{n} =0$ a.s., one deduces from the optional stopping theorem and Fatou's lemma that
\[
t\,\bbP(\tau_t <+\infty) \leq \bbE\left[M_{\tau_t}\indic_{\{\tau_t<+\infty\}}\right] = \bbE\Big[\lim_{n\to+\infty} M_{\tau_t\wedge n}\Big] \leq 1\,.
\]
This yields the upper bound.

Regarding the lower bound, it follows from a direct adaptation of the arguments that yield Proposition~\ref{prop:domination}. Let us assume that the lower bound in Proposition~\ref{prop_SD} does not hold: then for any $\gep>0$, there is some $t=t_{\gep}>1$ such that $\bbP(\sup_{n} M_n \geq t)  = \bbP(\tau_t <+\infty)\leq \gep t^{-1}$, similarly as in Remark~\ref{rem:tailtau}. But then one can repeat the arguments above to show that, for some $p$ sufficiently close to $1$, we have $\bbE[(M_n)^p] \leq 2 t^{p}$ for all $n\geq 1$, for some well-chosen (but fixed) $t>0$. Thus, the martingale $(M_n)_{n\geq 0}$ converges in $L^p$, contradicting the assumption that $M_{\infty}=0$ a.s.
\end{proof}

\subsection{Step 3. control of the overshoot: proof of Lemma~\ref{lem:oveshoot}}
\label{sec_overshoot}

We now prove Lemma~\ref{lem:oveshoot}.
We only have to prove that for any $k\geq 1$,
\[
  \bbE\big[  (M_k)^2  \indic_{\{\tau_t=k\}} \big] \leq C t^2\, \bbP(\tau_t=k) \,.
\]
For this, we decompose the expectation according to whether $M_k$ is large or not. 
Let $A>1$ be a fixed (large) constant, and write
\[
 \bbE[ (M_k)^2  \indic_{\{\tau_t=k\}}]  =  \bbE[ (M_k)^2  \indic_{\{\tau_t=k,M_{k}< At\}}] + \bbE[ (M_k)^2  \indic_{\{\tau_t=k,M_{k}\geq At\}}]  \,.
\]
The first term is obviously bounded by $(A t)^2\, \bbP(\tau_t=k)$, while for the second term, introducing the ratio $R_k:= M_k/M_{k-1}$ and conditioning by $\cF_{k-1}$, we have
\begin{equation}
\label{ineq:L2tau=k}
\bbE\big[ (M_k)^2  \indic_{\{\tau_t=k,M_{k}\geq At\}} \big]
\leq \bbE\Big[ (M_{k-1})^2 \indic_{\{\tau_t>k-1\}} \bbE\Big[ (R_k)^2 \indic_{\{R_k \geq \frac{At}{M_{k-1}}\}} \;\big|\; \cF_{k-1} \Big] \Big] \,.
\end{equation}
We are now reduced to the following lemma.
\begin{lemma}
\label{lem:ratio}
Let $\cR:= M_k/M_{k-1}$ and let us denote $\tilde\bbP:= \bbP( \;\cdot \mid \cF_{k-1})$.
There exists some $B_0$ and some constant $C_0$ such that, for any $B\geq B_0$ and any $k\geq 1$,
\[
\tilde \bbE\big[ \cR^2 \indic_{\{\cR \geq B\}}  \big] \leq C_0 B^2 \,\tilde\bbP( \cR \geq B)  \,.
\]
\end{lemma}

Using this lemma inside~\eqref{ineq:L2tau=k} with $A=B_0$, and noticing that $\frac{At}{M_{k-1}} \geq A\geq B_0$ on the event $\tau_t>k-1$, we obtain that
\[
\begin{split}
\bbE\big[ (M_k)^2 \indic_{\{\tau_t=k,M_{k}> At\}} \big]
& \leq C_0 \bbE\bigg[ (M_{k-1})^2 \indic_{\{\tau_t>k-1\}} \Big(\frac{At}{M_{k-1}}\Big)^2 \indic_{\{R_k \geq \frac{At}{M_{k-1}}\}} \bigg]\\
&= C_0 (A t)^2 \, \bbP\big(\tau_t>k-1 , M_k\geq At \big) \leq  C_0 (A t)^2 \bbP\big(\tau_t=k \big)  \,.
\end{split}
\]
All together, we obtain that for $A=B_0$, 
\[ 
\bbE\big[ (M_k)^2 \indic_{\{\tau_t=k\}} \big] \leq  (1+C_0) (A t)^2\, \bbP(\tau_t=k) \,,
\] 
which was our initial goal.
It therefore only remains to prove Lemma~\ref{lem:ratio}: we first give a general scheme of proof that relies on two claims; these claims are model-dependent and are proven afterwards.

\begin{proof}[Proof of Lemma~\ref{lem:ratio}]
First of all, and in order to lighten notation, we relabel the indices (replacing~$k$ by $k+1$), and we write
\begin{equation}\label{eq:def:cR:decomposition}
 \cR\coloneqq \frac{M_{k+1}}{M_k} = \sum_{z\in \partial^+_k} \alpha_z \check{M}_1(z)\,,
\end{equation}
where the second identity follows from the decomposition~\eqref{eq:decompoMn}; we also change the index in $\tilde \bbP(\cdot):= \bbP(\;\cdot\mid \cF_{k})$. 
Our goal is to show that, for all sufficiently large $B$, we have
\begin{equation}
\label{goal:ratio}
\tilde \bbE[\cR^2 \mid \cR\geq B] \leq C_0\, B^2 \,.
\end{equation}
The difficulty here is that, contrary to what happens in the directed polymer setting, the random variables $\check{M}_1(z)$ that appear in the above sum are \textit{not} independent conditionally on $\cF_k$.
Our strategy is to see $\cR$ as the limit of some martingale. To do that, we enumerate the elements of $\partial_k^+$, writing $\partial^+_k\eqqcolon\{z_n, n\geq1\}$, and for $n\geq0$ we let $\Lambda_n\coloneqq \bbS_k \cup \{z_1,\ldots,z_n\}$ and $\cG_n\coloneqq \sigma\{ \beta_{z},z\in \Lambda_n \}$ (in particular $\cG_0=\cF_k$). 
Finally, we define
\begin{equation}\label{def:Rn}
\cR_n \coloneqq \tilde \bbE[\cR\mid \cG_n] = \sum_{z\in \partial^+_k} \alpha_z \tilde \bbE[\check{M}_{1}(z)\mid\cG_n]\,,
\end{equation}
where the second equality uses the fact that the $\alpha_z$ are $\cF_k=\cG_0$-measurable.
(Note that no $\check M_1(z)$ is $\cG_n$-measurable, since $\check W_{x,y} >0$ for any $x,y \in \partial_k^+$, recall~\eqref{eq:def:checkM:1}.)

We now conclude the proof of the lemma subject to two estimates that are model-dependent.
Recall that $\tilde \bbP(\cdot) = \bbP(\;\cdot\mid \cF_{k})$ and $\cG_n$ have a dependence in $k$ that does not appear in the notation.

\begin{claim}
\label{claim:1}
There is a constant $c_1$ (uniform in $k$) such that, for any $n\geq 1$, 
\[
  \tilde \bbE[ \cR^2 \mid \cG_n] \leq c_1 (1+\cR_n^2)\,.
\]
\end{claim}

\begin{claim}
\label{claim:2}
Let $T := \inf\{ n\geq 0,\, \cR_n\geq 2 B\}$. Then, there is some constant $c_2$ (uniform in $k$) such that, for $B$ sufficiently large and $n\geq 1$,
\[
\tilde \bbE[\cR^2 \mid T=n] \leq c_1\big(1+\tilde \bbE[\cR_n^2 \mid T=n]\big) \leq c_2 B^2\,,
\] 
where $c_1$ is the constant appearing in Claim~\ref{claim:1}.
\end{claim}

For the VRJP, these estimates turn out to be quite technical and rely on the specific form of the $\beta$-potential: their proofs are postponed to Section~\ref{sec_claims} below.
However, they are easily verified in the context of the directed polymer (under some mild conditions) since in that case the variables $(\check{M}_{1}(z))_{z\in \partial_k^+}$ turn out to be i.i.d., as outline in the following remark.

\begin{remark}
  \label{rem:claimsforpolymer}
Consider the case where $\cR= \sum_{i=1}^{\infty} \alpha_i \xi_i$ with $\alpha_i\geq 0$ verifying $\sum_{i=1}^{\infty} \alpha_i \leq 1$ and $(\xi_i)_{i\geq 0}$ i.i.d.\ non-negative random variables with $\bbE[\xi_i]=1$ and $\bbE[\xi_i^2] <+\infty$ (this is the setting appearing in directed polymers).
Then, with the filtration $\cG_n:= \sigma\{ \xi_i, i\leq n\}$, one can easily verify that Claim~\ref{claim:1} holds without further condition and that Claim~\ref{claim:2} holds provided that $\bbE[ \xi_i \mid \xi_i \geq A ]\leq c A$ for some constant $c>0$, which corresponds to Condition~1 in~\cite{FJ23}.
Note that~\cite[Proposition~2.3]{JL24b} uses a slightly different proof for Lemma~\ref{lem:oveshoot}, with only a finite moments assumption for $\xi_i$ --- however, their proof, and in particular \cite[Eq.~(3.17)]{JL24b}, does not seem to be adaptable to the VRJP setting.
\end{remark}

We can now conclude the proof of Lemma~\ref{lem:ratio}.
First of all, since $\cR_n = \tilde \bbE[ \cR \mid \cG_n]$, applying Paley--Zygmund inequality, we obtain
\begin{equation*}
\label{eq:claim1}
\tilde \bbP\Big( \cR \geq \demi \cR_n \;\Big|\; \cG_n \Big) \geq \frac{1}{4} \frac{\bbE[\cR\mid\cG_n]^2}{\bbE[\cR^2 \mid \cG_n]} \geq \frac{1}{4 c_1} \frac{\cR_n^2}{1+\cR_n^2} \,,
\end{equation*}
thanks to Claim~\ref{claim:1}.
Now, since $ \frac12 \cR_n \geq B$ on the event $T=n$, this readily gives that  
\[
\tilde \bbP(\cR\geq B \mid T=n \big) \geq \frac{1}{4 c_1} \tilde \bbE\Big[\frac{\cR_n^2}{1+\cR_n^2} \;\Big|\; T=n \Big] \geq \frac{1}{4 c_1} \frac{B^2}{1+B^2} \geq \frac{1}{8c_1} \,,
\]
the last inequality holding as soon as $B\geq 1$.
We then deduce that
\[
\tilde \bbE\big[ \cR^2 \mid T=n, \cR\geq B\big]  = \frac{\tilde \bbE\big[ \cR^2 \indic_{\{\cR\geq B\}} \mid T=n\big]}{\tilde \bbP(\cR\geq B \mid T=n \big)}\leq 8c_1 \tilde \bbE\big[ \cR^2 \mid T=n \big]  \leq 8 c_1 c_2 \, B^2\,,
\]
the last inequality following from Claim~\ref{claim:2}.
All together, we obtain
\begin{align*}
\tilde \bbE\big[\cR^2 \mid \cR\geq B\big] &\leq B^2 + \sum_{n=1}^{+\infty} \tilde \bbE\big[\cR^2 \indic_{\{T=n\}} \mid \cR\geq B\big]  \\
& \leq  B^2 +  \sum_{n=1}^{+\infty} \tilde \bbP\big(T=n \mid \cR \geq B\big) \tilde \bbE\big[ \cR^2 \mid T=n, \cR\geq B \big]  \\
& \leq  B^2 + 8 c_1 c_2 B^2 \sum_{n=1}^{+\infty} \tilde \bbP\big(T=n \mid \cR \geq B\big) + B^2 \leq (1+8c_1 c_2) \, B^2 \,.
\end{align*}
This proves~\eqref{goal:ratio} and thus concludes the proof of Lemma~\ref{lem:ratio} --- assuming Claims~\ref{claim:1}-\ref{claim:2}.
\end{proof}

\section{Proof of Theorem~\ref{th:moment}: technical steps,  Claims~\ref{claim:1}-\ref{claim:2}}
\label{sec_claims}

As mentioned in Remark~\ref{rem:claimsforpolymer}, the proofs of Claims~\ref{claim:1}-\ref{claim:2} are easy when $\cR$ is a sum of i.i.d. random variables (under Condition~1 in \cite{FJ23}). 
There is no such independence in the case of the VRJP, however it satisfies other model-specific properties (see Section~\ref{sec_preliminaries}), most notably the fact that the law of the field $\gb$ is stable by taking marginals and conditional distributions, as stated in Lemma~\ref{lem:conditionalbeta}.

\subsection{Some useful estimates}
Recall from~\eqref{def:Rn} that we defined $\partial^+_k\eqqcolon\{z_n, n\geq1\}$, and for $n\geq0$, $\Lambda_n\coloneqq \bbS_k \cup \{z_1,\ldots,z_n\}$, $\cG_n\coloneqq \sigma\{\beta_{z},z\in \Lambda_n \}$ and $\cR_n=\bbE[\cR\mid \cG_n]$. 
Recall also the decomposition~\eqref{eq:def:cR:decomposition}, where $(\alpha_z)_{z\in \partial^+_k}$ are $\cG_0$-measurable and almost surely a probability distribution on $\partial^+_k$; recall also that $\check M_1$ is defined in~\eqref{eq:def:checkM:1}-\eqref{eq:def:checkM:2}.
In preparation to the proofs of Claims~\ref{claim:1}-\ref{claim:2}, let us collect some other properties of the VRJP.

\begin{proposition}\label{prop:VRJP:2ndmom}
  For all $n,k\geq0$, one has
  \begin{equation}\label{eq:prop:VRJP:2ndmom}
  \forall\, w\in \partial^+_k\setminus \Lambda_n,\qquad 
  \bbE\left[\check{M}_1(w)\,\big|\, \cG_n\right]\,=\, 1 \quad \text{ and } \quad 
  \bbE\left[\check{M}_1(w)^2\,\big|\, \cG_n\right]\,\leq\, 1+W^{-1}\,. 
  \end{equation}
  \end{proposition}

\begin{proof}
The first identity on the conditional expectation of $\check{M}_1(w)$ immediately follows from the definitions of Section~\ref{infinite-graph} (see~\eqref{eq:def:psin} in particular).

The second identity is a direct consequence of Theorem~\ref{thm_monotonicity}. 
Indeed, there exists a unique $w'\in \partial^+_{k+1}$ such that $w'\sim w$, so letting $\tilde W_{x,y}\coloneqq W\indic_{\{x,y\}=\{w,w'\}}\leq \check{W}_{x,y}$ for $x,y\in\bbS_{k+1}$, we obtain from Theorem~\ref{thm_monotonicity} that
\[
\bbE\left(\check M_1(w)^2 \;\big|\;\cG_n\right)\le \bbE\left(M^{\tilde W}_1(w)^2 \right) \,,
\]
where $M^{\tilde W}_1$ is the partition function of the polymer measure associated with the weights $\tilde W$, that is $M^{\tilde W}_1(w) = W \tilde\beta_{w}^{-1}\sim\IG(1,W)$ (recall Definition~\ref{beta}). 
Using standard calculations for the inverse Gaussian law, we get \(\bbE[M^{\tilde W}_1(w)^2] = 1+W^{-1}\), which concludes the proof.
\end{proof}

Let us also give a useful result when $z\in \partial_k^+ \cap \Lambda_n$.

\begin{proposition}
\label{prop:condGn}
For any $z\in \partial_k^+\cap \Lambda_n$, we have
\begin{equation}\label{eq:Mcheck1}
  \check{M}_1(z)\,=\, \sum_{w'\in \Lambda_n} G^{\Lambda_n}_\gb(z,w') \sum_{w\notin \Lambda_n} W_{w',w}   \left( \indic_{\{w \in \partial_{k+1}^+\}} + \check{M}_1(w) \indic_{\{w \in \partial_{k}^+\}}\right)\,,
\end{equation}
where the Green's functions $G^{\Lambda_n}_\gb$ is taken over paths that remain in $\Lambda_n$ and $W_{w',w} =W \indic_{w'\sim w}$.
As a consequence,
\begin{equation}\label{eq:condGn}
  \bbE\left[\check{M}_1(z) \; \big| \; \cG_n \right]\,=\,   \sum_{w'\in \Lambda_n} G^{\Lambda_n}_\gb(z,w') \sum_{w\notin \Lambda_n} W_{w',w}\,.
\end{equation}
\end{proposition}

\begin{proof}
Recall the formula~\eqref{eq:def:checkM:2} for $\check M_{1}(z)$. For $z\in \partial_k^+\cap\Lambda_n$, the sum over paths~$\gs$ can be decomposed on whether the path enters $\partial^+_k\setminus\Lambda_n$ or not, see Figure~\ref{fig:decoupe} for an illustration.
This yields
\begin{equation}\label{eq:proof:claim1:1}
  \check{M}_1(z)\,=\, \sum_{w'\in \Lambda_n} G^{\Lambda_n}_\gb(z,w') \sum_{w\in \partial_{k+1}^+} W_{w',w}  + \sum_{w'\in \Lambda_n, w\in \partial^+_k\setminus \Lambda_n} G^{\Lambda_n}_\gb(z,w') W_{w',w} \check{M}_1(w)\,,
\end{equation}
where the Green's functions $G^{\Lambda_n}_\gb$ is taken over paths that remain in $\Lambda_n$ and $W_{w',w} = W \indic_{w'\sim w}$.
Reorganizing the sums, we get~\eqref{eq:Mcheck1}.

\begin{figure}[ht]
  \begin{center}
    \includegraphics[width=.8\linewidth]{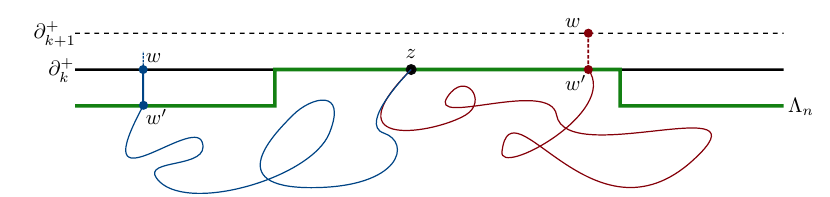}
    \caption{\footnotesize Illustration of the decomposition~\eqref{eq:proof:claim1:1}: we decompose the sum over paths in~\eqref{eq:def:checkM:2} according to whether the path stays inside $\Lambda_n$ before reaching $\partial_{k+1}^+$ (red/right path in the figure, first contribution in~\eqref{eq:proof:claim1:1}) or whether it first reaches $\partial^+_k\setminus \Lambda_n$ (blue/left path in the figure, second contribution in~\eqref{eq:proof:claim1:1}).}
      \label{fig:decoupe}
  \end{center}
\end{figure}

For~\eqref{eq:condGn}, we simply use~\eqref{eq:Mcheck1}, together with the fact that $\bbE[\check{M}_1(w) \mid \cG_n]=1$ for $w\in \partial_k^+ \setminus \Lambda_n$, thanks to Proposition~\ref{prop:VRJP:2ndmom}.
\end{proof}

\subsection{Proof of Claim~\ref{claim:1}}
Using Proposition~\ref{prop:VRJP:2ndmom} and Proposition~\ref{prop:condGn}, we get that
\[
  \check{M}_1(z)- \bbE\left[\check{M}_1(z)\,\big|\, \cG_n\right]
= \begin{cases}
  \  \check{M}_1(z) -1 &  \text{ if } z\in \partial_k^+\setminus \Lambda_n \,,\\
\displaystyle\sum_{w'\in \Lambda_n} G^{\Lambda_n}_\gb(z,w')\sum_{w\in \partial^+_k\setminus \Lambda_n} W_{w',w} (\check{M}_1(w)-1) & \text{ if } z \in \partial_k^+\cap \Lambda_n \,.
\end{cases}
\]
Therefore, splitting the sum $\cR-\cR_n = \sum_{z \in \partial_k^+} \alpha_z ( \check{M}_1(z)- \bbE[\check{M}_1(z)\mid \cG_n])$ according to whether $z\notin \Lambda_n$ or $z\in\Lambda_n$, we obtain that
\[
  \cR-\cR_n  =  \sum_{z\in \partial^+_k\setminus \Lambda_n} \ga_z\left(\check{M}_1(z)-1\right)  + \sum_{z\in \partial^+_k\cap\Lambda_n} \ga_z  \sum_{w'\in \Lambda_n, w\in \partial^+_k\setminus \Lambda_n} G^{\Lambda_n}_\gb(z,w') W_{w',w} \left(\check{M}_1(w)-1\right) \,.
\]
Rearranging the second sum, we can rewrite this as 
\[
  \cR-\cR_n  = \sum_{w \in \partial^+_k\setminus \Lambda_n}  \alpha_w^{(n)} \left(\check{M}_1(w)-1\right) 
  \  \text{ with } \ \alpha_w^{(n)} = \alpha_w + \sum_{z\in \partial^+_k\cap\Lambda_n , w'\in \Lambda_n} \ga_z  G^{\Lambda_n}_\gb(z,w') W_{w',w} \,.
\]
Then, applying Jensen's inequality (the $\alpha_w^{(n)}$ are non-negative), this yields
\[
  (\cR-\cR_n)^2 \,\leq\,  \bigg( \sum_{w\in \partial^+_k\setminus \Lambda_n} \ga_w^{(n)} \bigg) \bigg(\sum_{w\in \partial^+_k\setminus \Lambda_n} \ga_w^{(n)} \left(\check{M}_1(w)-1\right)^2 \bigg)\,,
\]
so that by Proposition~\ref{prop:VRJP:2ndmom} we get that 
\[
  \bbE\left[(\cR-\cR_n)^2\,\big|\, \cG_n\right] \leq (1+W^{-1}) \; \bigg( \sum_{w\in \partial^+_k\setminus \Lambda_n} \ga_w^{(n)} \bigg)^2
\]
recalling also that the $\alpha_w^{(n)}$ are $\cG_n$-measurable.
It remains to see that $\sum_{w\in \partial^+_k\setminus \Lambda_n} \ga_w^{(n)} \leq \cR_n$. 
But this comes from the same computation as above: 
indeed, using Propositions~\ref{prop:VRJP:2ndmom}-\ref{prop:condGn}, and decomposing the sum $\cR_n = \sum_{z \in \partial_k^+} \alpha_z  \bbE[\check{M}_1(z)\mid \cG_n]$ according to whether $z\notin\Lambda_n$ or $z\in \Lambda_n$, we get that
\begin{equation}
  \label{calculRn}
  \cR_n  =  \sum_{z\in \partial^+_k\setminus \Lambda_n} \ga_z  + \sum_{z\in \partial^+_k\cap\Lambda_n} \alpha_z \sum_{w'\in \Lambda_n, w\notin  \Lambda_n}  G_{\gb}^{\Lambda_n}(z,w') W_{w',w}\,,
\end{equation}
which is larger than the sum of the $\alpha_w^{(n)}$ (there is simply an extra term coming from the sum over $w\in \partial_{k+1}^+$, \textit{i.e.}\ $w\notin \Lambda_n \cup \partial_k^+$).

All together, we have proven that $\bbE[(\cR-\cR_n)^2\,|\, \cG_n] \leq (1+W^{-1})\cR_n^2$, which concludes the proof of Claim~\ref{claim:1}, with the constant $c_1 = 2+W^{-1}$.
\qed

\subsection{Proof of Claim~\ref{claim:2}}
Let us first observe that Claim~\ref{claim:2} can be deduced from Claim~\ref{claim:1} and the following lemma.

\begin{lemma}\label{lem:claim2proof}
There exists some (explicit) constant $c=c_{W}>0$ such that for all $B\geq 2$ and all $n\geq1$, one has a.s.\ 
\[ 
\bbE\left[\cR_{n}^2\indic_{\{\cR_{n}\geq B\cR_{n-1}\}} \; \middle| \; \cG_{n-1}\right] \leq c(B\cR_{n-1})^2 \, \bbP(\cR_{n}\geq B\cR_{n-1} \;|\; \cG_{n-1}) \,.
\]
\end{lemma}

\begin{proof}[Proof of Claim~\ref{claim:2} subject to Lemma~\ref{lem:claim2proof}]
The first inequality in Claim~\ref{claim:2} is an immediate consequence of Claim~\ref{claim:1}, since we easily have
\[
\tilde \bbE[\cR^2 \mid T=n] = \frac{\tilde \bbE[\tilde\bbE[\cR^2|\cG_n]\indic_{\{T=n\}}]}{\tilde\bbP(T=n)} \leq \frac{\tilde \bbE[c_1(1+\cR_n^2)\indic_{\{T=n\}}]}{\tilde\bbP(T=n)} = c_1\big(1+\tilde \bbE[\cR_n^2 \mid T=n]\big)\,.
\] 
Regarding the second inequality, let us show that $\tilde \bbE[ \cR_n^2 \mid T=n] \leq c' B^2$ for some constant~$c'$.
Note that since $\cR_0 =1$ $\tilde \bbP$-a.s.\ (this follows from~\eqref{eq:def:cR:decomposition} and the fact that $\bbE[\check M_1(z) \mid \cG_0]=1$ from~\eqref{eq:prop:VRJP:2ndmom}), one has $T\ge1$ a.s. 
Let us write, for $n\ge1$,
\[
\tilde\bbE\big[\cR_{n}^2\, \indic_{\{T=n\}}\big] = \tilde\bbE\left[\cR_{n}^2 \indic_{\{T=n\}}\Big(\indic_{\{\cR_{n}< \tilde B\cR_{n-1}\}} + \indic_{\{\cR_{n}\geq \tilde B\cR_{n-1}\}}\Big) \right] 
\]
where we set $\tilde B\coloneqq 2 B/\cR_{n-1}$. Then one has,
\begin{align*}
\tilde\bbE\big[\cR_{n}^2\, \indic_{\{T=n\}}\big]
&= \tilde\bbE\left[(\tilde B\cR_{n-1})^2 \indic_{\{T=n\}}\right] + \tilde\bbE\left[\cR_{n}^2 \indic_{\{T=n\}}\indic_{\{\cR_{n}\geq \tilde B\cR_{n-1}\}}\right]\\
&\leq (2B)^2\tilde\bbP(T=n) + \tilde\bbE\left[ \indic_{\{T>n-1\}}\bbE\left[\cR_{n}^2 \indic_{\{\cR_{n}\geq \tilde B\cR_{n-1}\}}\,\big|\,\cG_{n-1}\right]\right]\\
&\leq (2B)^2\tilde\bbP(T=n) + c\, \tilde\bbE\left[ (\tilde B\cR_{n-1})^2 \indic_{\{T>n-1\}} \indic_{\{\cR_{n}\geq \tilde B\cR_{n-1}\}} \right] \,,
\end{align*}
where the last inequality follows from Lemma~\ref{lem:claim2proof}, since $\tilde B\geq 2$ a.s.\ on the event $\{T>n-1\}$.
Then, noticing that 
\[
\{T>n-1\}\cap \{\cR_{n}\geq \tilde B\cR_{n-1}\} \,\subset\, \{T = n\}\,,
\]
this gives that $\tilde\bbE[\cR_{n}^2\, \indic_{\{T=n\}}] \leq 4(1+c)B^2 \tilde \bbP(T=n)$, which concludes the proof of the claim.
\end{proof}

\begin{proof}[Proof of Lemma~\ref{lem:claim2proof}]
Recall that $\Lambda_n = \Lambda_{n-1} \cup \{z_n\}$.
Using Proposition~\ref{prop:condGn} with $n-1$ in place of $n$, we get that for $z \in \partial_k^+ \cap \Lambda_{n-1}$,
\[
\check M_1 (z) = \sum_{w'\in \Lambda_{n-1}} G_{\beta}^{\Lambda_{n-1}}(z,w') \Big( \sum_{w\notin \Lambda_n} W_{w',w} \big(\indic_{\{w\in \partial_{k+1}^+\}} + \check M_1(w) \indic_{\{w \in \partial_k^+ \cap \Lambda_n\}} \big)  + W_{w',z_n} \check M_1(z_n)\Big) \,.
\]
Hence, using also that $\bbE[\check M_1 (z) \mid \cG_n ] =1$ for $z\in \partial_k^+\setminus \Lambda_n$ (thanks to Proposition~\ref{prop:VRJP:2ndmom}), we have for $z \in \partial_k^+ \cap \Lambda_{n-1}$
\[
  \bbE\left[\check{M}_1(z)\,\big|\,\cG_{n}\right] = \sum_{w'\in \Lambda_{n-1}} G_{\beta}^{\Lambda_{n-1}}(z,w') \Big( W_{w',z_n} \bbE\left[\check{M}_1(z_{n})\,\big|\,\cG_{n}\right] + \sum_{w\notin \Lambda_n} W_{w',w} \Big) \,.
\]
Hence, decomposing the sum $\cR_{n} = \sum_{z \in \partial_k^+} \alpha_z \check M_1(z)$ according to whether $z\in \Lambda_{n-1}$, $z=z_n$ or $z\notin \Lambda_n$, we obtain that a.s.
\begin{align*}
\cR_{n} &= \bigg(\sum_{z\in \Lambda_{n-1}} \ga_z \sum_{w'\in \Lambda_{n-1}} G^{\Lambda_{n-1}}_\gb(z,w') W_{w',z_{n}} + \ga_{z_{n}}\bigg) \bbE\left[\check{M}_1(z_{n})\,\big|\,\cG_{n}\right] \\
&\qquad \qquad \qquad  +\sum_{z\in \Lambda_{n-1}} \ga_z \sum_{w'\in \Lambda_{n-1}, w\notin \Lambda_{n} }G^{\Lambda_{n-1}}_\gb(z,w') W_{w',w} + \sum_{z\in \partial_k\setminus \Lambda_{n}} \ga_z\\
&\eqqcolon X_{n-1} \cdot Z_n + Y_{n-1}\,,
\end{align*}
where $Z_{n}\coloneqq \bbE[\check{M}_1(z_{n})\mid\cG_{n}]$, and $X_{n-1},Y_{n-1}$ are $\cG_{n-1}$-measurable and a.s.\ positive. In particular, one clearly has that $\cR_{n-1}=X_{n-1}+Y_{n-1}$ a.s. (recall also~\eqref{calculRn}).

Now, let $B>1$ and define $A\coloneqq B+(B-1)Y_{n-1}/X_{n-1}\ge B$, so that $A X_{n-1} +Y_{n-1} =B \cR_{n-1}$ and thus $\cR_{n} = X_{n-1} Z_n + Y_{n-1}\geq B\cR_{n-1}$ if and only if $Z_{n}\geq A$.
We therefore get that 
\begin{equation}
  \label{pfiuuu}
  \begin{split}
  \bbE\left[\cR_{n}^2\indic_{\{\cR_{n}\geq B\cR_{n-1}\}}\mid \cG_{n-1}\right] = 
\bbE\left[(X_{n-1}Z_n+Y_{n-1})^2\indic_{\{Z_{n}\geq A\}}\mid \cG_{n-1}\right] \,,
  \end{split}
\end{equation}
and expanding the square we need to control $\bbE[Z_n\indic_{\{Z_n\geq A\}}\mid\cG_{n-1}]$ and $\bbE[Z_n^2\indic_{\{Z_n\geq A\}}\mid\cG_{n-1}]$ in terms of $\bbP( Z_n\geq A \mid\cG_{n-1})$.

Recalling the properties of the VRJP from Lemma~\ref{lem:conditionalbeta}-Lemma~\ref{condionning_Hd} (and Remark~\ref{rem:condionning_Hd}), we notice by conditioning on $\cG_{n-1}$ that one can write
\begin{equation}
Z_n=\bbE\left[ \check{M}_1(z_{n}) \; \big| \; \cG_{n} \right] \eqqcolon \frac{\check{\eta}_{z_n} }{\check{\beta}_{z_n}}\,,
\end{equation}
where we set $\eta_x\coloneqq \sum_{y\in \partial_k^+\setminus\Lambda_{n}} W \indic_{\{y\sim x\}}$ for all $x\in\Lambda_n$, and
\begin{equation*}
\check\eta_{z_n}\,\coloneqq\, \eta_{z_n}+W_{\Lambda_{n-1},\{z_n\}}((H_\beta)_{\Lambda_{n-1},\Lambda_{n-1}})^{-1} \eta_{\Lambda_{n-1}}\,\ge\,W\,,
\end{equation*}
as in~\eqref{eq:lem:conditionalbeta}; and where, conditionally on $\cG_{n-1}$, $(\check{\beta}_{z_n})^{-1}$ has law $\nu^{0,\check{\eta}}_{\{z_n\}}$. In particular, one has $\cL(Z_n\mid \cG_{n-1}) = \IG(1, \check{\eta}_{z_n})$. 

We now use the following result on inverse Gaussian distributions, whose proof is postponed to the end of the section; it corresponds to verifying Condition~1 in~\cite{FJ23}.

\begin{lemma}\label{lem:claim2proof:IGlemma}
Let $\lambda_0>0$. There exists some constant $c=c_{\lambda_0}>1$ such that, for all $\lambda\ge\lambda_0$, and $A\ge 2$ and $Z\sim \IG(1,\lambda)$, one has
\begin{equation}
\bbE\left[Z^2\indic_{\{Z\geq A\}}\right] \leq c A^2\,\bbP(Z\geq A)\,.
\end{equation} 
\end{lemma}

Since one has $\check{\eta}_{z_n}\ge W$ a.s., this yields that
\begin{equation}
\label{eq:lem:claim2proof:IGlemma:1}
\bbE\left[Z_n^2\indic_{\{Z_n\geq A\}}\mid\cG_{n-1}\right]\,\leq\, cA^2\,\bbP(Z_n\geq A \mid \cG_{n-1})\,,
\end{equation} 
for a constant $c=c_{W}>0$ uniform in $A\geq 2$ and $n\geq1$. 
Furthermore, a conditional Cauchy--Schwarz inequality yields that
\begin{equation}\label{eq:lem:claim2proof:IGlemma:2}
\bbE\left[Z_n\indic_{\{Z_n\geq A\}}\mid \cG_{n-1}\right]\,\leq\, c^{1/2}\, A\,\bbP(Z_n\geq A \mid \cG_{n-1})\,.
\end{equation} 
Going back to \eqref{pfiuuu}, expanding the square and using~\eqref{eq:lem:claim2proof:IGlemma:1}-\eqref{eq:lem:claim2proof:IGlemma:2}, this yields for $A\geq 2$ (having $B\geq 2$ is enough)
\begin{align*}
\bbE\left[\cR_{n}^2\indic_{\{\cR_{n}\geq B\cR_{n-1}\}}\mid \cG_{n-1}\right] & \leq (c^{1/2} A X_{n-1}+Y_{n-1})^2 \bbP(Z_{n}\geq A \mid \cG_{n-1}) \\ 
&\leq c(B\cR_{n-1})^2 \bbP(\cR_{n}\geq B\cR_{n-1} \mid \cG_{n-1})\,,
\end{align*}
where we have also used that $AX_{n-1}+Y_{n-1} =  B \cR_{n-1}$ by definition.
This concludes the proof.
\end{proof}

\begin{proof}[Proof of Lemma~\ref{lem:claim2proof:IGlemma}]
Let $\lambda\ge\lambda_0$ and $Z\sim\IG(1,\lambda)$. Recall that the inverse Gaussian distribution $\IG(1,\lambda)$ has density
\[
\bbP(Z\in dt)\,=\, \frac{\sqrt{\lambda}}{\sqrt{2\pi}t^{3/2}} \exp\left(-\lambda\frac{(t-1)^2}{2t}\right)\indic_{\{t\ge 0\}} dt\,,\qquad t\in\R\,.
\]
In particular, one can show that for all $x\geq 2$ and $t\ge0$
\[
\bbP(Z\ge x+t)\,\leq\, e^{-t\lambda/4} \bbP(Z\ge x)\,\leq\, e^{-t\lambda_0/4} \bbP(Z\ge x)\,.
\]
Indeed, this follows from a direct computation and the observation that $\frac{(y+t-1)^2}{2(y+t)} \ge \frac{(y-1)^2}{2y}+\frac{3t}{8}$ for all $y\ge 2$ and $t\ge0$ (details are left to the reader). Therefore, for all $A\ge 2$ we have that
\begin{align*}
  \bbE\left[Z^2\indic_{\{Z\geq A\}}\right] = \sum_{k=1}^{\infty} \bbE\left[Z^2\indic_{\{Z\in [k A,(k+1)A]\}}\right] 
  & \leq \sum_{k=1}^{\infty} (k+1)^2 \bbP(Z \geq kA)  \\
  & \leq \sum_{k=1}^{\infty} (k+1)^2  e^{-(k-1) A\lambda_0/4}\bbP(Z\geq A) \,,
\end{align*}
which concludes the proof, with the constant $c_{\lambda_0}= \sum_{k=1}^{\infty} (k+1)^2  e^{-(k-1)\lambda_0/2}$.
\end{proof}

\section{Proof of Theorem~\ref{thm:allmoments}}
\label{sec_allmoments}

In this section, we focus on the dimension $d\geq 4$.
We consider the $\bbZ^d$-VRJP martingale defined in~\eqref{polymer_intro}; we recall its definition for the reader's convenience.
Consider the box $V_n = \llbracket-n,n\rrbracket^d\subset\bbZ^d$ and its outer boundary $\partial V_n$ and let
\begin{equation}
\label{def:Mn:allmoments}
\psi_{n} = \sum_{\sigma \colon 0\xrightarrow{V_n} \partial V_n} W^{|\sigma|} \prod_{i=1}^{|\sigma|-1} \beta_{\sigma_i}^{-1} \,,
\end{equation}
Then $(\psi_n)_{n\geq 0}$ is a martingale with respect to the filtration $\cF_n =\gs\{\gb_x,x\in V_n\}$; it converges a.s. towards some limit $\psi_\infty$ which characterizes the recurrence or transience of the VRJP in $\bbZ^d$ (recall Lemma~\ref{lem_Mmn}). 

Recall that we defined the critical point for the symmetric slab of width \(2m+1\):
\[
W_{c}^{(m)}\coloneqq W_c\left(\bbZ^{d-1}\times\llbracket -m,m\rrbracket\right) \,.
\]
Our goal is to show that if we fix some $m \in \bbN$ and some $W> W_c^{(m)}$, then the martingale $(\psi_n)_{n\geq 0}$ is bounded in $L^p$ for all $p\geq 1$.
The proof is performed in two steps: first, thanks to the monotonicity Theorem~\ref{thm_monotonicity}, we reduce the question to the study of a simpler graph; then, we control moments of order $p$ in this toy model.

\subsection{Comparison with a toy model}

Let us introduce a simplified graph, which is essentially one-dimensional, with a cemetery vertex.
It is represented in Figure~\ref{fig:simplifiedgraph} below.

\begin{definition}\label{def:simplegraph}
Let $n,m\ge0$, $\eps>0$ and $\mu_0$ some probability measure on $(0,+\infty)$. We define the graph $\tilde\cG=\tilde\cG_{\eps,\mu_0}^{n,m}$ with (random) weights $\tilde W$ by the following construction:
\begin{itemize}
  \item its vertices are given by $V\coloneqq\llbracket -n,n\rrbracket\cup\{\star\}$, where $\star$ is called the cemetery,
  \item for each $\{i,i+1\}\subset \llbracket -n,n\rrbracket$, there is an edge between $i$ and $i+1$ with weight $\tilde W_{i,i+1}\coloneqq\eps$. Similarly, the endpoints are connected with the cemetery with $\tilde W_{-n,\star}=\tilde W_{n,\star}\coloneqq\eps$,
  \item for each $i\in (2m+1)\bbZ$ such that $|i|\le n-m-2$, there is an edge between $i$ and $\star$ with weight $\tilde W_{i,\star}\coloneqq \tilde W_i$, where the $\tilde W_i$ are i.i.d.\ random variables with law $\mu_0$.
\end{itemize}
\begin{figure}[ht]\begin{center}\vspace{-4mm}
  \includegraphics[width=.8\linewidth]{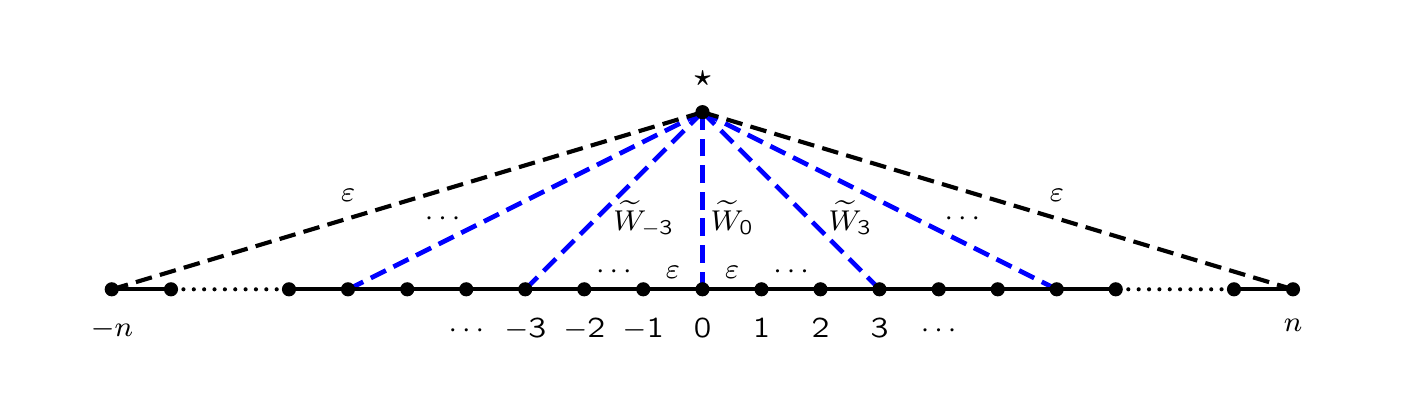}\vspace{-4mm}
  \caption{\footnotesize Illustration for the simplified graph $\tilde\cG$: in this picture $m=1$. The graph has vertex set $\llbracket -n,n\rrbracket\cup\{\star\}$. Each $i\in\llbracket -n,n\rrbracket$ is connected with its neighbours in $\llbracket -n,n\rrbracket$ by conductances $\eps$. 
  Moreover, the cemetery $\star$ is connected (represented with dashed lines) to the endpoints $-n$ and $n$ by conductances $\eps$, and to all vertices $i\in V\cap (2m+1)\bbZ$ which are not too close to the endpoints (i.e.\ $|i|\le n-m-2$) by i.i.d. conductances $\tilde W_i\sim \mu_0$ (in blue).
  }
  \label{fig:simplifiedgraph}
  \end{center}\end{figure}
\end{definition}

Then, we define the VRJP-polymer partition function associated with the toy model, as in~\eqref{polymer_intro}.
Let $n,m\ge0$, $\eps>0$ and~$\mu_0$ a probability measure on $(0,+\infty)$, and consider a realization of the random weighted graph $\tilde\cG\coloneqq \tilde\cG_{\eps,\mu_0}^{n,m}$. Conditionally on $\tilde\cG$, let $\gb\sim\nu^{\tilde W}_{\tilde\cG}$, and define
\begin{equation}\label{def:Mntilde:allmoments}
\tilde M = \tilde M\big(\tilde\cG\big) \coloneqq \sum_{\sigma \colon 0\xrightarrow{\llbracket-n,n\rrbracket} \{\star\}}  \tilde W_{\gs_0,\gs_1} \prod_{i=1}^{|\sigma|-1} \tilde W_{\gs_i,\gs_{i+1}} \beta_{\sigma_i}^{-1} \,.
\end{equation}

We prove in the next section that for $\gep$ sufficiently small (how small depends on $\mu_0$), the $p$-th moment of $\tilde M$ is bounded by a constant which does not depend on $n$, see Lemma~\ref{lem:allmoments:tildeG}.
But for now, let us compare the martingale~\eqref{def:Mn:allmoments} with $\tilde M$ on the graph $\tilde\cG$, with suitable parameters.

\begin{lemma}\label{lem:allmoments:comparison}
Let $n,m\ge0$ and set $2 \eps := W- W_{c}^{(m)}$. 
Then there exists a probability distribution~$\mu_{n}^{m,\eps}$ on $(0,+\infty)$ such that for any convex function $f:\R_+\to\R_+$, one has
\begin{equation}\label{eq:lem:allmoments:comparison}
\bbE\left[f(\psi_n)\right]\,\leq\, \bbE\left[f\left(\tilde M\left(\tilde\cG_{\eps,\mu_{n}^{m,\eps}}^{n,m}\right)\right)\right]\,,
\end{equation}
where $\tilde M$ is defined in~\eqref{def:Mntilde:allmoments};
we stress that the last expectation if over both $\beta$ and $\tilde W$ (\textit{i.e.}\ $\mu_{n}^{m,\eps}$).
\end{lemma}

\begin{proof}[Proof of Lemma~\ref{lem:allmoments:comparison}]
Let $n,m\ge0$, $\eps = \frac12 (W - W_{c}^{(m)})>0$ be fixed;  let $f:\R_+\to\R_+$ be a convex function. 
The strategy is to perform a series of graph modifications, which are illustrated in Figure~\ref{fig:graphcomparison} below. 
We define four consecutive graphs $\cG_j=(V_j,E_j)$, $0\le j\le3$, with a cemetery state~$\star$ (to lighten the notation we omit $\star$ in the vertex set $V_j$) and equipped with conductances $W=W^{(j)}=(W^{(j)}_e)_{e\in E_j}$.
The goal is to obtain a graph $\cG_j$ as a modification of the preceding one $\cG_{j-1}$, and such that the associated VRJP partition function 
\begin{equation}\label{eq:lem:allmoments:comparison:1}
M^{(j)} = \sum_{\sigma \colon 0\xrightarrow{V_j} \{\star\}} W^{(j)}_{\gs_0,\gs_{1}} \prod_{i=1}^{|\sigma|-1} W^{(j)}_{\gs_i,\gs_{i+1}} \beta_{\sigma_i}^{-1} \,,
\end{equation}
with $\gb\sim\nu^{W^{(j)}}_{\cG_j}$, verifies $\bbE\left[f(M^{(j)})\right]\leq \bbE\left[f(M^{(j+1)})\right]$ for all $0\le j<3$.

\begin{figure}[tbp]
  \begin{center}
  \begin{subfigure}{.4\linewidth}
    \begin{center}
      \includegraphics[width=\linewidth]{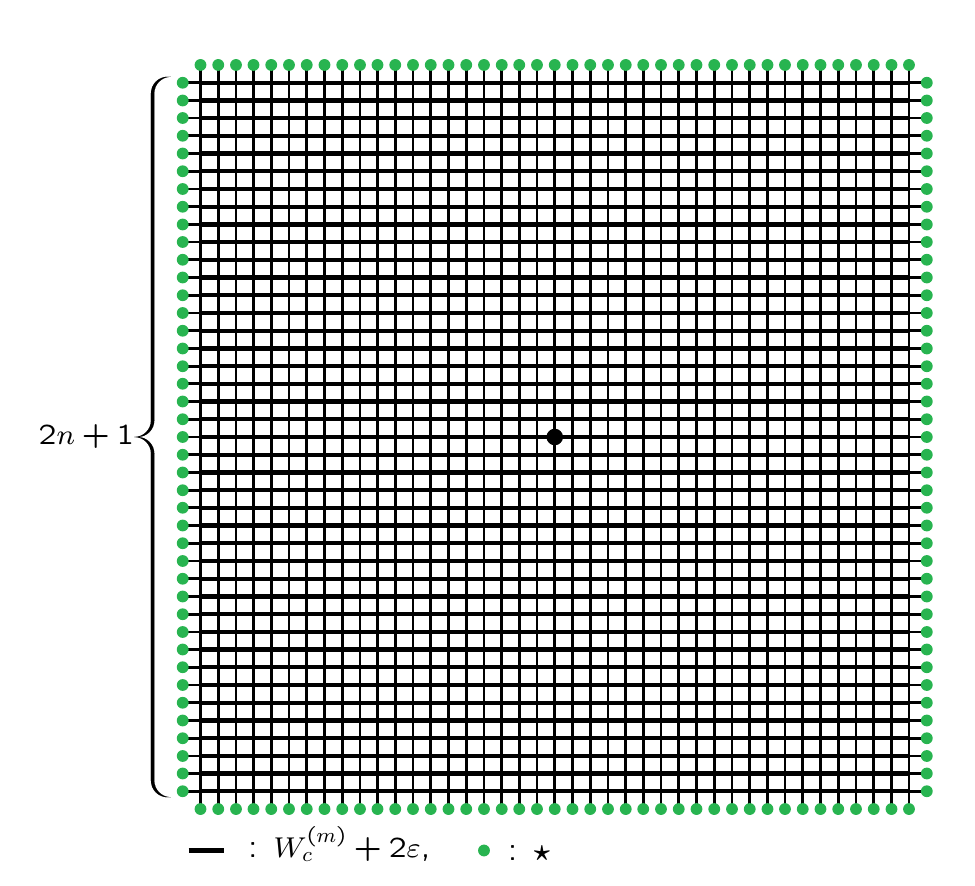}
      \caption{Graph $\cG_0$}
    \end{center}
  \end{subfigure}
  \qquad\qquad
  \begin{subfigure}{.4\linewidth}
    \begin{center}
      \includegraphics[width=\linewidth]{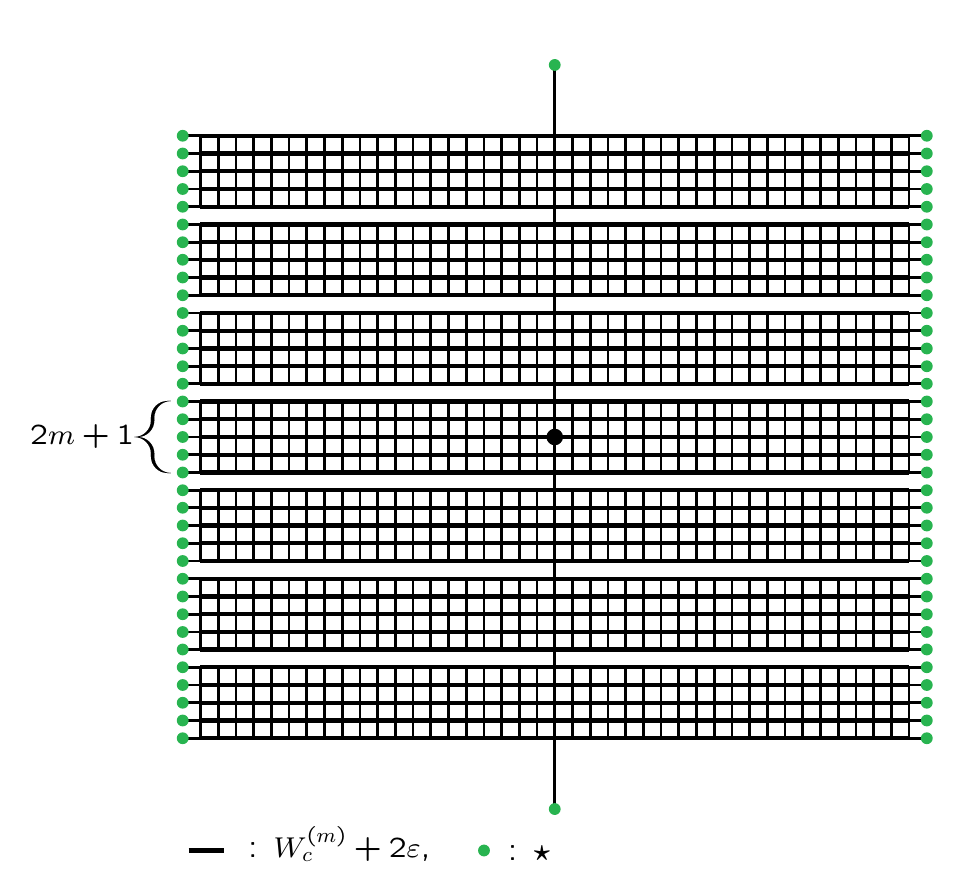}
      \caption{Graph $\cG_1$}
    \end{center}
  \end{subfigure}\\
  \begin{subfigure}{.4\linewidth}
    \begin{center}
      \includegraphics[width=\linewidth]{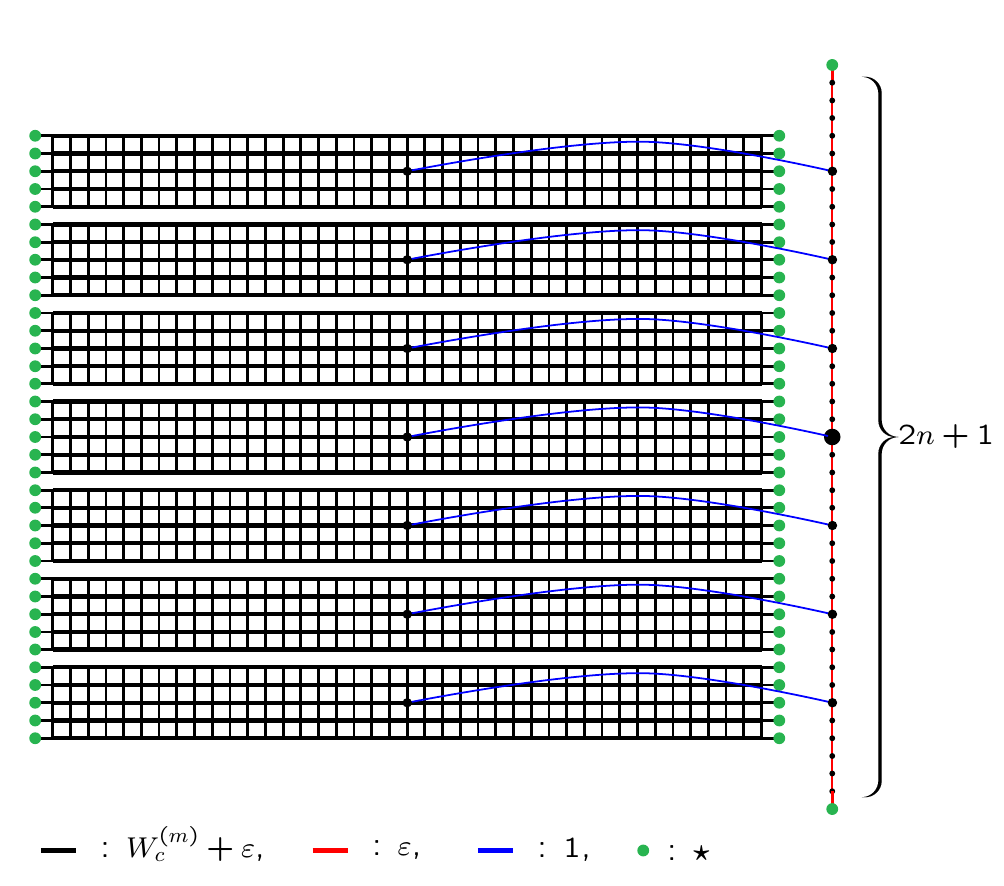}
      \caption{Graph $\cG_3$}
    \end{center}
  \end{subfigure}
  \qquad
  \begin{subfigure}{.4\linewidth}
    \begin{center}
      \includegraphics[scale=.4]{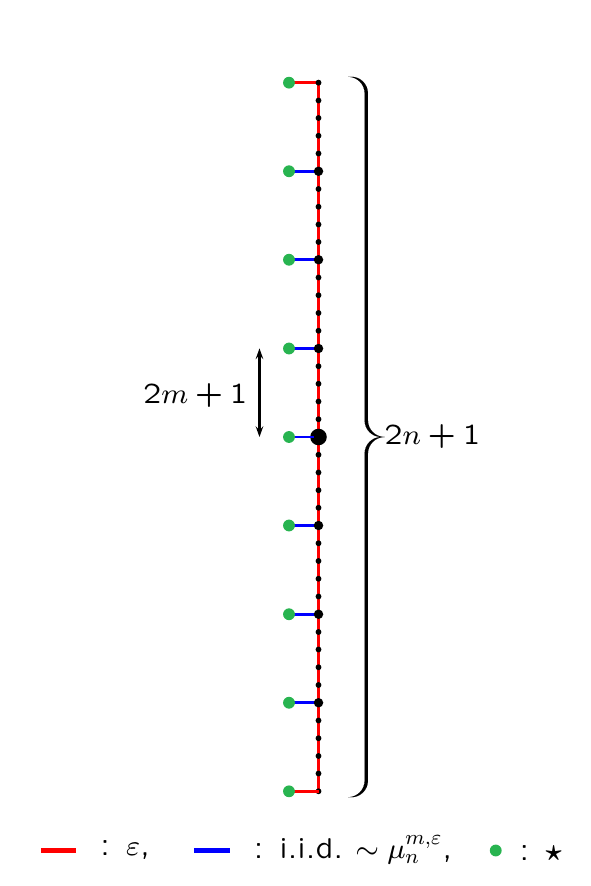}
      \caption{Graph $\tilde \cG$}
    \end{center}
  \end{subfigure}
\caption{\footnotesize Illustration of the different graphs used in the comparison: in each figure the green dots are identified with the cemetery $\star$ of the graph; the larger black dot represents the starting point of the VRJP. 
Step 0. The original graph on which $\psi_n$ is defined is denoted $\cG_0$, with vertex set $V_n=\llbracket-n,n\rrbracket^d$; all inner boundary vertices are connected to the cemetery $\star$, see~(A).
Step 1. The graph $\cG_1$ is obtained by removing edges so that the remaining ones form parallel slabs of width $2m+1$, connected along the vertical line $\{0\}^{d-1}\times \llbracket-n,n\rrbracket$, see~(B).
Step 2. The graph $\cG_2$ (not represented) is defined by duplicating the vertical line in $\cG_1$. 
Step 3. The graph $\cG_3$ is obtained by removing some edges and lowering the weights: its vertical line has conductances $\eps$, edges in the slabs have conductance $W_{c}^{(m)}+\eps$, the center of each slab is connect to the vertical line through an edge with conductance $1$, see~(C). 
Step 4. The graph $\cG_3$ is equivalent to a realization of the simple graph $\tilde\cG$ defined in Definition~\ref{def:simplegraph}, with blue edges having i.i.d.\ weight distribution $\mu_{n}^{m,\eps}$.
}
\label{fig:graphcomparison}
\end{center}\end{figure}

\medskip
\paragraph{\it Step $0$: initial graph.}
The graph $\cG_0$ is defined by the box $V_0\coloneqq \llbracket-n,n\rrbracket^d$, with conductances $W= W_c^{(m)} +2\eps$; this is illustrated in Figure~\ref{fig:graphcomparison}-A.
In particular, one has $M^{(0)}=\psi_n$ in law, where the equality in distribution of the fields $\gb$ follows from the restriction formula to $\llbracket-n,n\rrbracket^d$, recall Definition~\ref{beta_infinite}.
\medskip

\paragraph{\it Step $1$: edge removal.} 
The graph $\cG_1$ is obtained from $\cG_0$ by removing edges in such a way that the graph forms slabs isomorphic to $\llbracket-n,n\rrbracket^{d-1}\times\llbracket-m,m\rrbracket$, which are grafted to the vertical line $L\coloneqq\{0\}^{d-1}\times \llbracket-n,n\rrbracket$. This graph is illustrated in Figure~\ref{fig:graphcomparison}-B. Moreover, the inequality $\bbE\left[f(M^{(0)})\right]\leq \bbE\left[f(M^{(1)})\right]$ follows readily from Theorem~\ref{thm_monotonicity}.
\medskip

\paragraph{\it Step 2: vertical line duplication.} 
The next graph $\cG_2$ is obtained with a \emph{duplication} of the vertical line in $\cG_1$. More precisely, let us denote with $L\eqqcolon\{u_{-n},\ldots,u_n\}\subset V_1$: then we let $L'=\{u'_{-n},\ldots,u'_n\}$ be a copy of $L$, $V_2\coloneqq V_1\cup L'$, and we endow the graph with the conductances $W'$ defined as follows:
\begin{itemize}
  \item For $-n\leq i<n$, we let $W'_{u_i,u_{i+1}}\coloneqq W_{c}^{(m)}+\eps$ and $W'_{u'_i,u'_{i+1}}\coloneqq \eps$. 
  \item For any $u\in V_1$, $v\in V_1\setminus L$ with $\{u,v\}\in E_1$, we let $W'_{u,v}\coloneqq W_{c}^{(m)}+2\eps$ (\textit{i.e.}\ we keep the same conductance as in the graph $\cG_1$). 
  \item For any $i$, $|i|\le n$, we fix $W'_{u_i,u'_i}\coloneqq 1$.
\end{itemize}
Then, we claim that $\bbE\left[f(M^{(1)})\right]\leq \bbE\left[f(M^{(2)})\right]$. Indeed this follows from the fact that for any $|i|\leq n$, one can contract $\{u_i,u'_i\}\in E_2$ into a single vertex (this amounts to increasing the conductance in-between to $+\infty$): by the monotonicity principle~\cite[Theorem~5]{Poudevigne24} this transformation diminishes the quantity $\bbE\left[f(M)\right]$, with $M$ the VRJP partition function. Reproducing this contraction consecutively for each $i,|i|\leq n$, this transforms the graph $\cG_2$ back into $\cG_1$, and yields the expected inequality.
%
%
\medskip

\paragraph{\it Step 3: edge removal and weight decrease.} 
The graph $\cG_3$ is obtained by removing edges or lowering conductances in $\cG_2$. More precisely:
\begin{itemize}
  \item We remove edges in the original vertical line $L$ that connected different slabs: in particular, any path going from one slab to another has to pass through the new vertical line $L'$.
  \item We remove edges between $L'$ and the slabs that do not connect with the center of a slab.
  \item We lower the weights in the slabs so that they are all equal to $W_{c}^{(m)}+\eps$.
\end{itemize}
The resulting graph is illustrated in Figure~\ref{fig:graphcomparison}-C.
Here again, Theorem~\ref{thm_monotonicity} ensures that we have the inequality $\bbE\left[f(M^{(2)})\right]\leq \bbE\left[f(M^{(3)})\right]$.
\medskip

\paragraph{\it Step 4: equivalence to the toy graph $\tilde \cG$.}

We conclude the proof by showing that the VRJP on $\cG_3$ is equivalent to the VRJP on the simple graph $\tilde\cG$ from Definition~\ref{def:simplegraph}.
Indeed, let $U_3\coloneqq V_3\setminus L'$ denote the vertices of $\cG_3$ which are in one of the slabs (recall that this does not include $\star$), and let $\cF\coloneqq\gs\{\gb_x, x\in U_3\}$.
Then Lemma~\ref{lem:conditionalbeta} yields that, conditionally on $\cF$, the partition function $M^{(3)}$ on $\cG_3$ defined in~\eqref{eq:lem:allmoments:comparison:1} has the same law as on the realization of the graph $\tilde\cG$  with weights $\tilde W_i$ for $i\in (2m+1)\bbZ$, $|i|\le n-m-2$, which are given explicitly by the formula~\eqref{eq:lem:conditionalbeta}. 

In particular, those weights have the same law as the martingale $\psi_n^{(m)}$ of the VRJP on a box $V_n^{(m)} = \llbracket-n,n\rrbracket^{d-1}\times\llbracket-m,m\rrbracket$ started from its center and with boundary term
$W_{c}^{(m)}+\eps$ on the lateral faces of the box (recall that the top and bottom faces of each slab in $\cG_3$ are not connected to $\star$). Moreover the slabs in $\cG_3$ are at distance at least $2m+3$ from one another, so this also implies that those weights $\tilde W_i$ are independent. 
More precisely, the weights $\tilde W_i$ are i.i.d.\ with the same distribution as 
\[
\psi_{n}^{(m)} =  \sum_{\sigma \colon 0 \xrightarrow{V_n^{(m)}} V^{(m)}\setminus V_n^{(m)}}
(W_c^{(m)}+\eps)^{|\sigma|}\prod_{i=0}^{\vert \sigma\vert -1} \beta_{\sigma_i}^{-1} \,,
\]
where $V^{(m)}  =\bbZ^{d-1} \times \llbracket-m,m\rrbracket$ is the $m$-slab.
Therefore, taking the expectation, this yields
\[
\bbE\left[f(M^{(3)})\right]\,=\, \bbE\left[f\left(\tilde M\left(\tilde\cG_{\eps,\mu_{n}^{m,\eps}}^{n,m}\right)\right)\right]\,,
\]
where $\mu_{n}^{m,\eps}$ is the law of $\psi_{n}^{(m)}$.
This concludes the proof of~\eqref{eq:lem:allmoments:comparison}. 
\end{proof}

We also have the following claim, which control the convergence of the distribution~$\mu_{n}^{m,\eps}$.
\begin{lemma}
  \label{lem:nondegenerate}
For $m,\eps$ fixed, the sequence $(\mu_{n}^{m,\eps})_{n\ge0}$ admits a non-degenerate limit $\mu_{\infty}^{m,\eps}$ as $n\to+\infty$, i.e.\ $\mu_{\infty}^{m,\eps}$ is a probability measure on $(0,+\infty)$.
\end{lemma}

\begin{proof}
This follows immediately from Theorem~\ref{Th_martingale}: indeed, the VRJP on $V^{(m)}=\bbZ^{d-1}\times\llbracket-m,m\rrbracket$ with weights $W=W_{c}^{(m)}+\eps>W_{c}^{(m)}$ is transient, so its martingale $\psi_n^{(m)}$ admits an a.s.\ positive limit when $n\to+\infty$, and we denote its law with $\mu_{\infty}^{m,\eps}$.
\end{proof}

\subsection{Boundedness in \texorpdfstring{$L^p$}{Lp} for the toy model}

Let us state the following result on the toy model, from which we deduce easily Theorem~\ref{thm:allmoments}.
Recall the Definition~\ref{def:simplegraph} of the graph $\tilde \cG$ and recall~\eqref{def:Mntilde:allmoments} for its associated partition function~$\tilde M$. 

\begin{lemma}
  \label{lem:allmoments:tildeG}
  Let $\eps>0$ and $m\in \bbN$ be fixed.
  Then, for all $p\geq 1$ there exists some $\eps_0 = \eps_0(m,\eps,p)$ such that, as soon as $\eta_0 \leq \eps$ verifies $\mu_0((0,\eta_0))<\eps_0$ there is a constant $C=C(m,\eps,\eps_0,\eta_0,p)<+\infty$ such that 
  \[
    \sup_{n\geq 1} \bbE\big[(\tilde M)^p \big]\le C \,.
  \]
  Here again, $\bbE$ is the expectation over both $\gb$ and $\tilde W$. 
  In particular, $C$ does not depend on $n\ge0$ or the details of $\mu_0$.
  \end{lemma}
  
Let us first conclude the proof of Theorem~\ref{thm:allmoments} before we turn to that of Lemma~\ref{lem:allmoments:tildeG}.

\begin{proof}[Proof of Theorem~\ref{thm:allmoments}]
  Let $W>W_c^{\rm slab}$ and $p\geq 1$.
  In particular, there is some $m \in \bbN$ such that $W>W_c^{(m)}$, and we let $\eps \coloneqq \frac12 (W-W_c^{(m)}) >0$.
  Our goal is to show that $\sup_n \bbE[(\psi_n)^p]<+\infty$, but by Lemma~\ref{lem:allmoments:comparison}, this will follow if we prove that
\[
    \sup_n \bbE\left[\left(\tilde M\left(\tilde\cG_{\eps,\mu_{n}^{m,\eps}}^{n,m}\right)\right)^p\right]\,<\,+\infty\,.
\]
  We let $\eps_0$ be as in Lemma~\ref{lem:allmoments:tildeG} and also $\eta_0>0$ be such that $\sup_n\{\mu_{n}^{m,\eps}((0,\eta_0)), n\ge0\}< \eps_0$; the fact that $\eta_0$ may be chosen independently of $n$ follows from Lemma~\ref{lem:nondegenerate}. 
Then, Lemma~\ref{lem:allmoments:tildeG} yields that the expectation above is bounded by some finite constant $C = C(m,\eps,p)$ uniformly in $n\ge0$, which concludes the proof.
\end{proof}

\begin{proof}[Proof of Lemma~\ref{lem:allmoments:tildeG}]
  Let $K \coloneqq \inf\{  j \geq 0, \tilde W_{ (2m+1)j,\star} \geq  \eta_0 \}$, so that by assumption $K$ is dominated by a geometric random variable with parameter $1-\eps_0$.

  Then, conditionally on $K=k$, we can compare $\tilde \cG$ with a one-dimensional graph $\tilde \cG_{\ell\wedge n}$ with $\ell \coloneqq (2m+1)k$.
  If $\ell <n$, we construct $\tilde \cG_\ell$ from $\tilde \cG$ by keeping only the edges $\{i,i+1\}$ for $0\leq i \leq \ell$, with weight $\tilde W_{i,i+1} =\eps$, and also the extra edge $\{\ell, \star\}$, reducing its weight to $\eta_0$.
  If $(2m+1)k\geq n$, we consider only $\tilde \cG_n$ as above (we also reduce the weight of the edge $\{n,\star\}$ from $\eps$ to~$\eta_0$).
  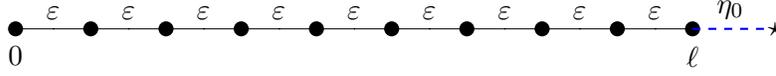
\begin{figure}
    \begin{center}
      \begin{tikzpicture}
        \draw[step=0.5,black,thin] (0,0) grid (9,0);
        \foreach \x in {0,1,2,3,4,5,6,7,8} { 
          \fill[color=black] (\x,0) circle (0.1);
          \node[above] at (\x+0.5,0) {$\eps$};
        }
        \node[below] at (0,-0.1) {$0$};
        \fill[color=black] (9,0) circle (0.1);
        \node[below] at (9,-0.1) {$\ell$};
        \draw[dashed, thick, blue] (9,0) -- (10,0);
        \node[above] at (9.5,0) {$\eta_0$};
        \node at (10.1,0) {$\star$};
      \end{tikzpicture}
      \caption{\footnotesize Illustration of the graph $\tilde \cG_{\ell}$, here with $\ell = 9$.}\label{figure_4}
      \end{center}
      \vskip-0.3cm
  \end{figure}


Finally, let $\tilde M_\ell$ the VRJP partition function on $\tilde \cG_\ell$, \textit{i.e.}\ 
\[
  \tilde M_\ell  = \sum_{\sigma : 0 \xrightarrow{\llbracket 0, \ell \rrbracket}{\star}}  \eta_0 \, \eps^{|\sigma|-1}\prod_{i=1}^{|\sigma|-1} \gb_{\sigma_i}^{-1} \,.
\]
Then, conditioning on the value of $K$ and using the monotonicity property (Theorem~\ref{thm_monotonicity}), we get 
\begin{align*}
  \bbE\big[ (\tilde M)^p \big] & \leq \sum_{k=0}^{ n/(2m+1) } \bbP\big( K= k \big) \bbE\big[ (\tilde M_{(2m+1)k})^p \big] + \bbP(K \geq n/(2m+1) ) \bbE\big[ (\tilde M_n)^p \big] \,.
\end{align*}
Note that by a direct computation (we postpone the proof until the end of this section), we have that 
  \begin{equation}
    \label{momentk}
    \bbE\big[(\tilde M_k)^p \big] = \bbE[(Y_{\eps})^{p}]^{(2m+1)k} \, \bbE[(Y_{\eta_0})^p] \,.
  \end{equation}
where 
$Y_{a} \sim \mathrm{IG}(1,a)$. In particular, we can write that there are constants $C_0 = C_0(m,\eps,p)= \bbE[Y_{\eps}^{p}]^{(2m+1)}$ and $c_1 = c_1(\eta_0,p) = \bbE[(Y_{\eta_0})^p]$ such that\footnote{There is no need here, but we could bound very roughly $\bbE[(Y_{a})^p] \leq \frac{2^p}{a^p} p! \bbE[e^{\frac{a}{2} Y_a}] = \frac{2^p p!}{a^p} e^a $, where we have used the Laplace transform of the inverse Gaussian distribution in the last identity.} 
$\bbE\big[(\tilde M_k)^p \big] = c_1 C_0^k$.
Therefore, since we have by assumption that $\bbP( K= k)  \leq \bbP(K\geq k) = \eps_0^k$, we obtain
\[
  \bbE\big[ (\tilde M)^p \big] \leq c_1 \sum_{k=0}^{ +\infty } \eps_0^k C_0^k  \leq 2 c_1 \,,
\]
where we have chosen $\eps_0 =\frac12 C_0^{-1}$ for the last inequality; recall that $c_1$ still depends on $\eta_0$.

It therefore only remains to prove~\eqref{momentk}. It is a classical result,
see e.g. \cite[Theorem~3]{Chen-Zeng18} which gives the statement if we express the polymer $\tilde M_\ell$ in terms of the mixing field of the VRJP starting at the vertex $\star$. Let us however give a self-contained proof of the result using the $\beta$-potential. We consider the 1-dimensional graph with length $\ell+1$ given in Figure~\ref{figure_4}. 
Let $A_0, \ldots, A_{\ell}$ be independent random variables with law $A_{i}\sim\mathrm{IG}(1,\epsilon)$ for \(0\leq i \leq \ell-1\) and $A_\ell\sim \mathrm{IG}(1,\eta_0)$. 
Define now $(\tilde \beta_{i})_{i=0, \ldots, \ell}$ by
$$
\begin{cases}
\tilde \beta_0 = \frac{\epsilon}{A_0} \,,&
\\
\tilde \beta_i= {\epsilon A_{i-1}}+\frac{\epsilon}{A_{i}},& 1\le i\le \ell-1 \,,
\\
\tilde \beta_\ell= {\epsilon A_{\ell-1}}+\frac{\eta_0}{A_{\ell}}\,.&
\end{cases}
$$
Since the Laplace transform of the inverse Gaussian law and its inverse satisfy, for $Y_a\sim \mathrm{IG}(1,a)$,
$$
\bbE(e^{-\demi \lambda  a Y_a})=e^{-a(\sqrt{1+\lambda}-1)}, \;\;\; \bbE(e^{-\demi \lambda a \frac{1}{Y_a}})=\frac{e^{-a(\sqrt{1+\lambda}-1)}}{\sqrt{1+\lambda}} \,,
$$
we see that $(\tilde \beta_{i})_{i=0, \ldots, \ell}$ has the same law as the potential $\beta_{|\{0, \cdots, \ell\}}$ for the graph of Figure~\ref{figure_4}, recalling the expression~ \eqref{eq:laplace-nubetaweta} for its Laplace transform.

Besides, simple computations show that $\psi_{\tilde \ggg_\ell} (x)$ given by
$\psi_{\tilde \ggg_\ell}(\star)=1$
and  $\psi_{\tilde \ggg_\ell}(x)=\prod_{i=x}^\ell A_i$
solves the equation $(H_\beta \psi_{\tilde \ggg_\ell})(x)=0$ on that weighted graph, for any $x=0, \ldots, \ell$. 
Hence it implies that $\psi_{\tilde \ggg_\ell} (x)$ is the partition function of polymers starting at $x$ and endpoint~$\star$. We conclude that
$$
\tilde M_\ell=\psi_{\tilde \ggg_\ell}(0)= \prod_{i=0}^\ell A_i \,,
$$
which proves~\eqref{momentk} since the $(A_i)_{0\leq i\leq \ell}$ are independent with the correct inverse Gaussian laws.
\end{proof}



\appendix
\section{Proof of the final statement of Lemma~\ref{lem_Mmn}}\label{appendix_A}

Recall the notations of Lemma~\ref{lem_Mmn}. It remains to prove that $M_{n,\infty}=M_n$. The idea is to decompose the martingale $M_{n,m}$ according to the contributions of the polymers which exit the box $\bbB_{n,m}$ on its side or on the top, and to relate them to the exit probabilities of the VRJP.
We set
$$
\partial_{n,m}^0=\partial_{n,m}^+\cap \partial_{n}^+=\{x\in \partial_{m,n}^+, \; x_d=n\}, \;\;\; \partial_{n,m}^1=\partial_{n,m}^+\setminus \partial_{n}^+=\{x\in \partial_{n,m}^+, \; x_d<n\},
$$
respectively the top and side parts of the boundary of the box $\bbB_{n,m}$. Accordingly, we define, for $i=0,1$,
\begin{equation}\label{M_mn_}
M_{n,m}^i :=\sum_{\sigma \colon 0\xrightarrow{\bbB_{n,m}} \partial^i_{n,m}} W^{|\sigma|} \prod_{i=1}^{|\sigma|-1} \beta_{\sigma_i}^{-1}.
\end{equation}
the contribution in $M_{n,m}$ of the polymers that exit on the top or the side of $\bbB_{n,m}$. 
Hence, we want to prove that $\lim_{m\to \infty} M_{n,m}^1=0$ a.s.
From the proof of Lemma~\ref{lem_Mmn}, we know that $M_{n, m}$ is bounded in $L^p$, so that $M_{n,\infty}>0$ a.s. and the statement is equivalent to prove that $\lim_{m\to \infty} \frac{M_{n,m}^1}{M_{n,m}}=0$ a.s. The lemma below relates this ratio to the exit probability of the VRJP.
\begin{lemma}\label{M_n_VRJP}
We have
\begin{align}\label{proba_sortie}
\bbE\left( \frac{M_{n,m}^1}{M_{n,m}}\right) =P^{VRJP}_0 \left( 0\xrightarrow{\bbB_{n,m}} \partial^1_{n,m} \right),
\end{align}
where the right-hand side denotes the probability that the VRJP starting at 0 exits the box $\bbB_{n,m}$ by the side boundary $\partial^1_{n,m}$.
\end{lemma}
\begin{proof}
Consider the graph $\tilde \cG_{n,m}=(\tilde \bbB_{n,m}, \tilde E_{n,m})$ obtained by contracting all the vertices of $\partial_{n,m}^+$ to a single point $\star$, \textit{i.e.}\ by taking a wired boundary condition on $\bbB_{n,m}$. Hence, the vertex set is $\tilde \bbB_{n,m}=\bbB_{n,m}\cup\{\star\}$ and $\star$ is connected to each point of the interior boundary of $\bbB_{n,m}$ by an edge with conductance $W$. By the restriction Lemma~\ref{lem:conditionalbeta}~(i), $\beta_{\bbB_{n,m}}$ has the same law under $\nu_{\tilde \bbB_{n,m}}^W$ and under $\nu_{\bbH_d}^W$.

For simplicity, we now denote by $\beta=(\beta_{i})_{i\in \tilde B_{n,m}}$ the random $\beta$-potential on the graph $\tilde \cG_{n,m}$ and by $G_\beta$ its Green function.  Consider now the VRJP on the weighted graph $\tilde \cG_{n,m}$. By the representation~\eqref{VRJP-rep}, the VRJP starting at $0$, is, up to a time change, a mixture of Markov Jump Processes with jump rate
$$
\indic_{i\sim j} W \frac{G_\beta(0,j)}{G_\beta(0,i)}.
$$
In particular, if we denote by $(\tilde Y_n)$ the associated discrete time process, then it is a mixture of Markov chains with transition probabilities
\begin{eqnarray}\label{MC-rep}
\frac{1}{\tilde \beta_i} \indic_{i\sim j} W \frac{G_\beta(0,j)}{ G_\beta(0,i)} \qquad
\text{ where } \tilde \beta_i= \sum_{j, j\sim i} W \frac{G_\beta(0,j)}{G_\beta(0,i)}.
\end{eqnarray}
Remark that by definition of $G_\beta$ we have $\tilde \beta_i=\beta_i$ if $i\neq 0$.

Denote by $P^{\beta}_0$ the law of the Markov chain on the graph $\tilde \cG_{n,m}$ with transition probabilities given by~\eqref{MC-rep}, and, for $i=0,1$, let
$$
P^{\beta}_0\Big( 0\xrightarrow{\bbB_{n,m}} \partial^i_{n,m} \Big)
$$
be the probability that the Markov chain reaches the boundary vertex $\star$ by $\partial^0_{n,m}$ or~$\partial^1_{n,m}$. Then by definition of $P^{\beta}_0$ we have,
$$
P^{\beta}_0\Big( 0\xrightarrow{\bbB_{n,m}} \partial^i_{n,m} \Big) = 
\frac{G_\beta(0,\star)}{G_\beta(0,0)} \sum_{\sigma \colon 0\xrightarrow{\bbB_{n,m}} \partial^i_{n,m}} W^{|\sigma|} \prod_{i=1}^{|\sigma|-1} \tilde \beta_{\sigma_i}^{-1}.
$$
In the last sum, the contribution of the loops from 0 do not depend on the choice of the boundary $\partial^0_{n,m}$ or $\partial^1_{n,m}$. 
We get 
\begin{align*}
P^{\beta}_0\Big( 0\xrightarrow{\bbB_{n,m}} \partial^1_{n,m} \Big)
&= \frac{P^{\beta}_0\Big( 0\xrightarrow{\bbB_{n,m}} \partial^1_{n,m} \Big)}{P^{\beta}_0\Big( 0\xrightarrow{\bbB_{n,m}} \partial^0_{n,m} \Big)+ P^{\beta}_0\Big( 0\xrightarrow{\bbB_{n,m}} \partial^1_{n,m} \Big)}
\\
&= \sum_{\sigma \colon 0\xrightarrow{\bbB_{n,m}} \partial^1_{n,m}} W^{|\sigma|} \prod_{i=1}^{|\sigma|-1} \tilde \beta_{\sigma_i}^{-1}
\Bigg\slash
\sum_{\sigma \colon 0\xrightarrow{\bbB_{n,m}} \partial^+_{n,m}} W^{|\sigma|} \prod_{i=1}^{|\sigma|-1} \tilde \beta_{\sigma_i}^{-1} \,.
\end{align*}
Besides, considering that $\beta_i=\tilde\beta_i$ for $i\neq 0$, we see that \(P^{\beta}_0\big( 0\xrightarrow{\bbB_{n,m}} \partial^1_{n,m} \big)= \frac{M_{n,m}^1}{M_{n,m}}\).
\end{proof}

The last step is to prove that the right-hand side of \eqref{proba_sortie} goes to $0$ when $ m$ tends to $+\infty$. Consider $(Y_t)_{t\ge 0}$ the VRJP on $\bbH_d$ starting at 0. We denote by $\tau_{n,m}$ the exit hitting time of $\bbB_{n,m}$,
$$
\tau_{n,m}=\inf\{t\ge 0, \;\;\; Y_t\in \partial_{n,m}^+\},
$$
and by $A^1_{n,m}$ the event that it exits on the side of the box
$$
A^1_{n,m}=\{ Y_{\tau_{n,m}} \in   \partial_{n,m}^1\}.
$$
We also set $\cF^Y_{n,m}=\sigma\{Y_t, \;\; t\le \tau_{n,m}\}$. We prove that there exists some $c>0$, depending only $W$, $d$ and $n$, such that for all $m\ge 1$,
\begin{eqnarray}\label{lower_prob}
\bbP^{VRJP}_{0}\left( (A^1_{n,m+2})^c \; \vert \; A^1_{n,m}\right)\ge c.
\end{eqnarray}
which is enough to end the proof.

Remark from the definition of the VRJP, see \eqref{def_VRJP}, that the jump rate from $x\in \partial_{n,m}^1$ to the set $\partial_{n,m+1}^+$ is bounded from below by $W$ (since there is at least one edge from any $x\in \partial_{n,m}^1$ to the set $\partial_{n,m+1}^+$). It implies that, conditionally on $\cF^{Y}_{n,m}$ and on the event $A_{n,m}^1$, the local time 
$$
L_{\partial_{n,m}^1}(\tau_{n,m+1}):= \sum_{x\in \partial_{n,m}^1} L_x(\tau_{n,m+1})
$$ 
is dominated by an exponential random variable of parameter $W$. Hence, there is a constant $c_0>0$ such that $\bbP^{VRJP}_{0}( L_{\partial_{n,m}^1}(\tau_{n,m+1})\le 1 \; \vert \; A^1_{n,m})\ge c_0$. Besides, conditionally on $L_{\partial_{n,m}^1}(\tau_{n,m+1})\le 1$ and on the event $A^1_{n,m+1}$, the probability that the VRJP after the time $\tau_{n,m+1}$ makes only jumps in the up direction $e_d$ until it reaches $\partial_{n,m+1}^0$ is bounded from below by a constant. The last event is included in $(A^1_{n,m+2})^c$, which implies the inequality \eqref{lower_prob}.

\hfill\break
\noindent
{\bf Acknowledgment.}
This work is supported by the  project  ANR LOCAL (ANR-22-CE40-0012-02)  operated by the Agence Nationale de la Recherche (ANR).

\bibliographystyle{abbrv}
\bibliography{biblio.bib}

\end{document}